\documentclass[fleqn,11pt]{article}
\usepackage{graphics}
\usepackage{color}
\usepackage{mathrsfs}
\usepackage{amssymb}
\usepackage{amsmath}
\usepackage{boxedminipage}
\usepackage{stmaryrd}
\usepackage{multirow}
\usepackage{booktabs}
\usepackage{accents}

\begin{document}

\pagenumbering{arabic}
\newcounter{comp1}

\newtheorem{definition}{Definition}
\newtheorem{proposition}{Proposition}
\newtheorem{example}{Example}
\newtheorem{method}{Method}
\newtheorem{lemma}{Lemma}
\newtheorem{theorem}{Theorem}
\newtheorem{corollary}{Corollary}
\newtheorem{assumption}{Assumption}
\newtheorem{algorithm}{Algorithm}
\newtheorem{remark}{Remark}
\newcommand{\fig}[1]{\begin{figure}[hbt]
                  \vspace{1cm}
                  \begin{center}
                  \begin{picture}(15,10)(0,0)
                  \put(0,0){\line(1,0){15}}
                  \put(0,0){\line(0,1){10}}
                  \put(15,0){\line(0,1){10}}
                  \put(0,10){\line(1,0){15}}
                  \end{picture}
                  \end{center}
                  \vspace{.3cm}
                  \caption{#1}
                  \vspace{.5cm}
                  \end{figure}}
\newcommand{\tabincell}[2]{\begin{tabular}{@{}#1@{}}#2\end{tabular}}%
\newcommand{\Axk}{A(x^k)}
\newcommand{\Aumb}{\sum_{i=1}^{N}A_{i}u_{i}-b}
\newcommand{\Kk}{K^k}
\newcommand{\Kki}{K_{i}^{k}}
\newcommand{\Aukmb}{\sum_{i=1}^{N}A_{i}u_{i}^{k}-b}
\newcommand{\Au}{\sum_{i=1}^{N}A_{i}u_{i}}
\newcommand{\Aukpmb}{\sum_{i=1}^{N}A_{i}u_{i}^{k+1}-b}
\newcommand{\nab}{\nabla^2 f(x^k)}
\newcommand{\xk}{x^k}
\newcommand{\ubk}{\overline{u}^k}
\newcommand{\uhk}{\hat u^k}
\def\QEDclosed{\mbox{\rule[0pt]{1.3ex}{1.3ex}}} 
\def\QEDopen{{\setlength{\fboxsep}{0pt}\setlength{\fboxrule}{0.2pt}\fbox{\rule[0pt]{0pt}{1.3ex}\rule[0pt]{1.3ex}{0pt}}}} 
\def\QED{\QEDopen} 
\def\proof{\par\noindent{\em Proof.}}
\def\endproof{\hfill $\Box$ \vskip 0.4cm}
\newcommand{\RR}{\mathbf R}

\title {\bf
Auxiliary Problem Principle of augmented Lagrangian with Varying Core Functions for Large-Scale Structured Convex Problems
             }
\author{Lei Zhao\thanks {Antai College of Economics and Management and Sino-US Global Logistics Institute, Shanghai Jiao Tong University, 200030 Shanghai, China({\tt l.zhao@sjtu.edu.cn})}, Daoli Zhu\thanks {Antai College of Economics and Management and Sino-US Global Logistics Institute, Shanghai Jiao Tong University, 200030 Shanghai, China({\tt dlzhu@sjtu.edu.cn})}, Bo Jiang\thanks {Research Center for Management Science and Data Analytics, School of Information Management and Engineering,
Shanghai University of Finance and Economics, Shanghai 200433, China({\tt isyebojiang@gmail.com})}}
\maketitle

\begin{abstract}
The auxiliary problem principle of augmented Lagrangian (APP-AL), proposed by Cohen and Zhu (1984), aims to find the solution of a constrained optimization problem through a sequence of auxiliary problems involving augmented Lagrangian.
The merits of this approach are two folds. First, the core function is usually separable, which makes the subproblems at each step decomposable and particularly attractive for parallel computing. Second, the choice
of the core function is quite flexible. Consequently, by carefully specifying this function,
APP-AL may reduce to
some standard optimization algorithms.
In this paper, we pursue enhancing such flexibility by allowing the core function to be non-identical at each step of the algorithm, and name it
varying auxiliary problem principle (VAPP-AL). Depending on the problem structure, the varying core functions in VAPP-AL can be adapted to design new flexible and suitable algorithm for parallel and distributed computing. The convergence and $O(1/t)$ convergence rate of VAPP-AL for convex problem with coupling objective and constraints is proved. Moreover, if this function is specialized to be quadratic, an $o(1/t)$ convergence rate can be established. Interestingly, the new VAPP framework can cover several
variants of Jacobian type augmented Lagrangian decomposition methods as special cases. Furthermore, our technique works for the convex problem with nonseparable objective and multi-blocks coupled linear constraints, which usually can not be handled by ADMM.

\vspace{1cm}


%

\end{abstract}
\normalsize
\newpage
\vspace{1cm}
\section{Introduction}
In this paper, we consider the following block structured convex minimization problem with nonseparable objective and coupled linear equality constraints:
\begin{equation}\label{Prob:general-function}
\begin{array}{lll}
 \mbox{(P):} & \min     & (G+J)(u)=G(u_1,u_2,\cdots,u_N)+\sum_{i=1}^{N}J_{i}(u_{i})      \\
             & \rm {s.t}& Au=\sum_{i=1}^{N}A_{i}u_{i}=b \\
             &          & u_{i}\in U_{i},\quad i=1,2,\ldots,N.
\end{array}
\end{equation}
\noindent where $G$ is a convex smooth function on $U\subset \RR^n$, $U=U_1\times\cdots\times U_{N}$ each $J_{i}$ is a convex but possibly nonsmooth function on $U_{i}\subset \RR^{n_{i}}$, and $A=(A_{1},A_{2},\cdots,A_{N})\in \RR^{m\times n}$ is an appropriate partition of matrix $A$ and $A_{i}$ is an $m\times n_{i}$ matrix, $b\in \RR^{m}$ is a vector. When the coupling term $G$ is absent from the objective, problem (P) become to a separable programming where the objective functions and constraint functions can be expressed as the sum of the each involving one block variable.
\begin{equation}
\begin{array}{lll}
 \mbox{(SP):}  &\min     & \sum_{i=1}^{N}J_{i}(u_{i})      \\
               &\rm {s.t}& Au=\sum_{i=1}^{N}A_{i}u_{i}=b \\
               &         & u_{i}\in U_{i},\quad i=1,2,\ldots,N
\end{array}
\end{equation}
Four decades ago, Hestenes~\cite{Hestenes1969} and Powell~\cite{Powell1969} introduced augmented Lagrangian methods (ALM), also named as multipliers method, for equality constrained nonlinear program. Theoretical properties of augmented Lagrangian duality method was investigated by Rockafeller~\cite{Rock76}. ALM for problem (P) are based on its associated augmented Lagrangian function:
\begin{equation}\label{func:AL}
L_{\gamma}(u,p)=(G+J)(u)+\langle p,Au-b\rangle+\frac{\gamma}{2}\|Au-b\|^{2},\quad \mbox{with}\quad J(u)=\sum_{i=1}^{N}J_{i}(u_{i}), Au=\sum_{i=1}^{N}A_iu_i.
\end{equation}
In particular, it can be described as follows.\\
\noindent\rule[0.25\baselineskip]{\textwidth}{1.5pt}
{\bf Augmented Lagrangian method (ALM)}\\
\noindent\rule[0.25\baselineskip]{\textwidth}{0.5pt}
{Initialize} $u^0 \in U$ and $p^0\in \mathbf{R^m}$  \\
 \textbf{for} $k = 0,1,\cdots $, \textbf{do}
\begin{eqnarray}
u^{k+1}&\leftarrow&\min_{u\in U}L_\gamma (u,p^k);\label{UZAWA_primal}\\
p^{k+1}&\leftarrow&p^{k}+\rho (Au^{k+1}-b).\label{UZAWA_dual}
\end{eqnarray}
\textbf{end for}\\
\noindent\rule[0.25\baselineskip]{\textwidth}{1.5pt}
Note the variables $u_i, i=1,\cdots, N$ are coupled in the nonlinear function $G(u)$ and quadratic term $\frac{\gamma}{2}\|\sum_{i=1}^{N}A_iu_i-b\|^2$ , hence the augmented Lagrangian function $L_{\gamma}$ is nonseparable.
In the following, we shall review a few approaches that can overcome this challenge, so that the primal problem~\eqref{UZAWA_primal} of ALM can break into the $N$ independent subproblems of $u_i$ which is easy to solve.
\begin{itemize}
\item{} {\bf Coupling linearization and regularization methods}\\
Cohen and Zhu~\cite{CohenZ} and Zhu~\cite{Zhu83} proposed Auxiliary problem principle of augmented Lagrangian (APP-AL) to solve (P):\\
\noindent\rule[0.25\baselineskip]{445pt}{1.5pt}
{\bf Auxiliary problem principle of augmented Lagrangian method (APP-AL)}\\
\noindent\rule[0.25\baselineskip]{445pt}{0.5pt}
{Initialize} $u^0 \in U$ and $p^0\in\mathbf{R}^m$  \\
 \textbf{for} $k = 0,1,\cdots $, \textbf{do}
\begin{eqnarray}
u^{k+1}&\leftarrow&\min_{u\in U}\langle \nabla G(u^{k}), u \rangle + J(u)
+ \langle p^{k}+\gamma(Au^{k}-b), Au\rangle \nonumber\\
&&+\frac{1}{\epsilon}[K(u)- \langle \nabla K(u^{k}), u \rangle];\label{APP_primal}\\
p^{k+1}&\leftarrow&p^{k}+\rho (Au^{k+1}-b).\label{APP_dual}
\end{eqnarray}
\textbf{end for}\\
\noindent\rule[0.25\baselineskip]{445pt}{1.5pt}
In APP-AL algorithm, a core function $K(u)$ is introduced. The objective function of ~\eqref{APP_primal} is obtained by  keeping the separate part $J(u)$,
linearizing the coupling part $G(u)$ as well as the quadratic term $\frac{\gamma}{2}\|\sum_{i=1}^{N}A_iu_i-b\|^2$ in augmented Lagrangian and adding a regularization term $\frac{1}{\epsilon}[K(u)-K(u^k)-\langle\nabla K(u^k),u\rangle]$ which is essentially the Bregman distance function. If the core function $K$ is separable, i.e., $K(u)=\sum_{i=1}^{N}K_i(u_i)$, so as the subproblem~\eqref{APP_primal}.
Thanks to this decomposable property, excellent numerical performance can be achieved and APP-AL has become the main theoretical basis of some parallel computing software such as DistOpt~\cite{Distopt1, Distopt2}. Besides, APP also has wide applications in engineering systems, such as power systems~\cite{Electro93,Electro97}, multiple-robot systems~\cite{CelaHamam92}. In particular, this approach was adopted by Kim and Baldick to parallelize optimal power flow in very large interconnected power systems~\cite{Electro97,KimBaldick00}; \c{C}ela and Hamam applied this method to solve the optimal control problem of multiple-robot systems in the presence of obstacles~\cite{CelaHamam92}. The main advantage of coupling linearization and regularization methods is that it leads to decompositions on both objective and constraint, and is thus suitable for parallel implementation. Chen and Teboulle~\cite{ChenTeboulle94} also proposed linearied proximal method of multipliers LPMM to solve SP with two blocks and linear convergence rate of this algorithm has been proved. Note that LPMM is exactly
APP-AL for (SP) by taking $K(u)=\frac{\|u\|^2}{2}$.
\item{}{\bf Gausse-Seidel type method}\\
Fortin M. and Glowinski~\cite{FortinGlowinski} and Glowinski~\cite{Glowinski84} consider the following (SP) with two blocks
\begin{equation}
\begin{array}{lll}
 \mbox{(SP$_2$):}  &\min\limits_{u\in U,v\in \mathbf{V}}     & F(u)+G(v)      \\
                   &\rm {s.t}                       & Bu-v=0
\end{array}
\end{equation}
and proposed ALG$_2$ which is also called Alternating direction method of multiplier (ADMM).
\noindent\rule[0.25\baselineskip]{445pt}{1.5pt}
{\bf ALG$_2$, Alternating direction method of multiplier (ADMM)}\\
\noindent\rule[0.25\baselineskip]{445pt}{0.5pt}
{Initialize} $u^0 \in U$ and $p^0\in \mathbf{R^m}$  \\
 \textbf{for} $k = 0,1,\cdots $, \textbf{do}
\begin{eqnarray}
u^{k+1}&\leftarrow&\min_{u\in U}F(u)+\langle p^{k}, Bu\rangle +\frac{\gamma}{2}\|Bu-v^k\|^2;\label{ADMM_u}\\
v^{k+1}&\leftarrow&\min_{v\in \mathbf{V}} G(v)-\langle p^{k}, v\rangle +\frac{\gamma}{2}\|Bu^k-v\|^2;\label{ADMM_v}\\
p^{k+1}&\leftarrow&p^{k}+\gamma (Bu^{k+1}-v^{k+1}).\label{ADMM_dual}
\end{eqnarray}
\textbf{end for}\\
\noindent\rule[0.25\baselineskip]{445pt}{1.5pt}
This method apply the Gausse-Seidel like approach to overcome the coupling quadratic term $\frac{\gamma}{2}\|Bu-v\|^2$; thus allowing for decomposition of the minimization in $u$ and $v$ respectively. Noted that in the ADMM we need to assume that $B$ has full column rank. Shefi and Teboulle~\cite{ShefiTeboulle14} proposed Alternating direction proximal method of multipliers (Algorithm 2, AD-PMM) without assumption of matrix. Chen et al.~\cite{ChenHeYeYuan13} shown a coutre example indicating that the multi-blocks ADMM ($N\geq 3$) may divergence to problem (SP). Some works to investigate further sufficient conditions on the problem which can guarantees convergence for the multi-blocks ADMM. The existing results typically required the strongly convexity for function $J_i(u_i)$ in the objective. Noted that in this situation, one can directly use Lagrange method rather than augmented Lagrangian method. Recently,~\cite{CuiLiSunToh-2015} proposed the perturbation modified ADMM to get $\varepsilon$-optimization do not require strongly convexity of the objective function. \cite{LiMoYuanZhang14} investigated the algorithm of mixture of ADMM and coupling linearization and regularization to solve problem (P) with two blocks. In their works, the coupling function $G$ is linearized, Gausse-Seidel step are used to decouply the quadratic term in the augmented Lagrangian function.\\
Since 2007, there was a resurgence of multiplier method, it renewed interest in ADMM has emerge from new applications arising in signal and image processing, machine learning, and other fields, which often modeled by structured and very large scale convex optimization problems (P) or (SP), See Boyd et al.'s excellent review paper~\cite{Boyd2011}.
\item{}{\bf Jacobian type method}\\
Jacobian is another scheme to overcome the difficulty of coupling quadratic term of multiplier method for problem (SP). The advantage of the Jacobian type methods is all the decomposed subproblems are eligible for parallel computation at each iteration as the linearization-regularization approach. To guarantee the convergence of full Jacobian decomposition, He~\cite{HeHouYuan13} take a relaxation step at every iteration and assume the matrix $A_i$ have full column-rank, the algorithm and convergence rate $O(1/t)$. Deng~\cite{DengLaiPengYin-2014} assume that $A_i$ are near orthogonal, and have full column-rank, and proposed to add proximal terms of different kinds to subproblems to guarantee the algorithm convergence globally at a rate of $o(1/t)$.
\end{itemize}
The goals of this paper is outlined as following. We extend the classical APP-AL decompositions of Cohen and Zhu~\cite{CohenZ} to accommodate the varying core functions in each iteration and is referred to as VAPP-AL. This approach can serve as a framework of decomposition algorithm to solve generic convex optimization with nonseparable objective and multi-blocks linear equality constraints. Depending on the problem structure, the varying core functions in VAPP-AL can be adapted to design new flexible and suitable algorithm for parallel and distributed computing. (See the examples in Section~\ref{sec:example}, various linearization/regularization, Gausse-Seidel, Jaccobian or Mixture type algorithms are introduced to handle practical problems). In this work, global convergence and iteration complexity of VAPP-AL are discussed as well.\\
The rest of this paper is organized as follows. We start Section 2 with some
notations and assumptions that we make in this paper. After that, VAPP-AL framework is officially proposed followed. Then we show how to specify this framework to solve
numerous practical problems. In Section 3 we prove the global convergence for VAPP-AL. Then the $O(1/t)$ convergence rate is established in Section 4 both in ergodic and non-ergodic sense. In Section 5,
we show that the convergence rate of VAPP-AL can be improved to $o(1/t)$ when the core function is specialized to be
quadratic. Finally, we end our paper with some conclusions.

\section{Varying Auxiliary Problem Principle of augmented Lagrangian (VAPP-AL)}
\subsection{Preliminaries}
In this paper, we let
$$
u := \left(\!\! \begin{array}{c} u_1 \\ \vdots \\ u_N \end{array}\!\! \right) \in \RR^n,\;U=U_1 \times \cdots \times U_N\;\mbox{and}\;A=[A_1,\cdots, A_N] \in \RR^{m \times n},\;\mbox{where}\;n = \sum_{i=1}^{N}n_i.
$$
We denote $\langle \cdot \rangle $ and $\| \cdot \|$ as the inner product and Euclidean norm of vector, respectively. For a matrix $B \in \RR^{\ell \times \ell}$,  $\|B\|$ strands for the spectral norm,
which is the largest singular value of $B$. We use $\lambda_{\max}(B)$ and $\lambda_{\min}(B)$ to denote the maximum and minimum eigenvalue of $B$ respectively.

For a positive semi-definite matrix $Q \in \RR^{\ell \times \ell}$, the semi-norm associated with $Q$ is denoted by $\|\cdot\|_Q$. In particular, for any vector $x$,
$\|x\|_Q = \sqrt{x^T Q x}$.

\noindent Throughout this paper, we make the following standard assumptions:
\begin{assumption}\label{assump1}
{\rm
\noindent \begin{itemize}
\item[(i)] $J$ is a convex, l.s.c function (not necessary differentiable) such that $\mathbf{dom}J\cap U\neq \emptyset$.
\item[(ii)] $G$ is a convex and differentiable with its derivative Lipschitz of constant $B_{G}$.
\item[(iii)] $G+J$ is coercive on $U$, if $U$ is not bounded, that is
\begin{eqnarray*}
\forall \{u^{k}|k\in N\}\subset U, \lim_{k\rightarrow+\infty}\|u^{k}\|=+\infty\Rightarrow\lim_{k\rightarrow+\infty}(G+J)(u^{k})=+\infty.
\end{eqnarray*}
\item[(iv)]\begin{eqnarray}0\in \mbox{interior of } \big{(}A(U)-b\big{)}. \label{zero-inpoint}\end{eqnarray}
\end{itemize}
}
\end{assumption}
The Lagrangian function of problem~\eqref{Prob:general-function} is given by
$$
L(u,p)=(G+J)(u)+\langle p,Au-b\rangle.
$$
The pair $(u^{*},p^{*})\in U\times R^{m}$ is called a saddle point of Lagrangian function $L(u,p)$ if it holds that
\begin{equation}
L(u^{*},p)\leq L(u^{*},p^{*})\leq L(u,p^{*}),\; \forall u\in U,\; \forall p\in \RR^{m}. \label{eq:3}
\end{equation}
The saddle point optimality condition for problem (P) is given by
\begin{theorem}
A solution $(u^*,p^*)$ with $u^*\in U$ and $p^*\in \mathbf{R}^m$ is a saddle point for the Lagrangian function $L(u,p)$ if and only if
\rm{
\begin{itemize}
\item[(i)] $L(u^*,p^*)=\min_{u\in U}L(u,p^*)$;
\item[(ii)] $Au^*-b=0$;
\item[(iii)] $\langle p^*, Au^*-b\rangle=0$.
\end{itemize}
}
\end{theorem}
It is easy to get the VI reformulation:
\begin{eqnarray}
 && (G+J)(u)-(G+J)(u^*)+\langle p^*, A(u-u^*)\rangle\geq 0, \forall u\in U  \label{VIS_1}    \\
 && \langle Au^*-b,p-p^*\rangle\leq 0, \forall p\in R^m \label{VIS_2}
\end{eqnarray}
\noindent Note that (iii) and (iv) in Assumption \ref{assump1} are imposed to guarantee the existence of saddle point of $L$ on $U\times \RR^{m}$.
Moreover, under Assumption \ref{assump1}, $L$ and $L_{\gamma}$ share the same sets of saddle points.
The interested readers are referred to the book of Bertsekas~\cite{Bert} and the paper of Cohen and Zhu~\cite{CohenZ} for more results on Lagrangian duality theory.

\subsection{VAPP-AL method for (P)}
Based on the augmented Lagrangian theory, in this section we propose new primal-dual parallel algorithm to solve (P).
To extend the augmented Lagrangian decomposition method of Cohen and Zhu~\cite{CohenZ}, we introduce the core function $K^k(\cdot)$ that varies in different iterations and  satisfies the following assumption:
\begin{assumption}\label{assump2}
$\Kk(\cdot)$ is strongly convex with parameter $\beta^{k}$ and differentiable with its gradient Lipschitz continuous with parameter $B^{k}$ on $U$. Moreover, there exist positive numbers $\beta$ and $B$ such that $0<\beta\leq \beta^{k}\le B^{k} \le B$, for any $k\in \mathbb{N}$.
\end{assumption}
When the above assumption holds, function $D_K^k(u,v)=K^k(u)-K^k(v)-\langle \nabla K^k(v),u-v\rangle$ is a distance like function which has uniform lower and upper bound: $\frac{\beta}{2}\|u-v\|^2\leq D_K^k(u,v)\leq\frac{B}{2}\|u-v\|^2$.
Then for given $u^k$, we obtain the approximation of augmented Lagrangian $L_\gamma(u,p)=(G+J)(u)+\langle p, Au-b\rangle+\frac{\gamma}{2}\|Au-b\|^2$ as follows:
\begin{eqnarray*}
\tilde{L}_\gamma(u,p)=G(u^k)+\langle \nabla G(u^k),u-u^k\rangle+J(u)+\langle p, Au-b\rangle+\frac{\gamma}{2}\|Au^k-b\|^2+\gamma\langle Au^k-b, A(u-u^k)\rangle
+\frac{1}{\epsilon}D_K^k(u,u^k).
\end{eqnarray*}
Now we are ready to present VAPP-AL for (P):\\
\noindent\rule[0.25\baselineskip]{\textwidth}{1.5pt}
{\bf Varying Auxiliary Problem Principle (VAPP-AL)}\\
\noindent\rule[0.25\baselineskip]{\textwidth}{0.5pt}
{Initialize} $u^0 \in U$ and $p^0\in \mathbf{R}^m$  \\
 \textbf{for} $k = 0,1,\cdots $, \textbf{do}
\begin{eqnarray}
u^{k+1}&\leftarrow&\min_{u\in U}\langle\nabla G(u^{k}), u \rangle + J(u)+ \langle p^k+\gamma(Au^k-b), Au\rangle+\frac{1}{\epsilon}D_K^k(u,u^k);\label{primal}\\
p^{k+1}&\leftarrow&p^{k}+\rho (Au^{k+1}-b).\label{dual}
\end{eqnarray}
\textbf{end for}\\
\noindent\rule[0.25\baselineskip]{\textwidth}{1.5pt}
Note that problem \eqref{primal} is essentially equivalent to
\begin{eqnarray}
\mbox{(AP$^k$)}\qquad\min_{u\in U}\langle\nabla G(u^{k}), u \rangle + J(u)+ \langle p^k+\gamma(Au^k-b), Au\rangle+\frac{1}{\epsilon}\big{[}K^k(u)-\langle\nabla K^k(u^k),u\rangle\big{]}.\label{eq:APk}
\end{eqnarray}
\subsection{Various Implementations of VAPP-AL and Applications to Practical Problem}\label{sec:example}
The core function $K^k$ plays an important role in VAPP-AL. Depending on the structure of the problem, $K^k$ can be constructed properly to make subproblems decomposable and thus easy to solve. Typically,
$K^k$ is composed by three parts: (i) a variant of coupled function $G(u)$ in the objective; (ii) a variant of quadratic penalty on the linear coupled constraint $\frac{\gamma}{2}\|\sum_{i=1}^NA_iu_i-b\|^2$; (iii) a quadratic regularized term $\frac{1}{2}\sum_i^{N} \| u_i\|^2_{M_i}$.
Table~\ref{table1} provides some options for the three parts that constitute the core function. Hence any combination of these choices could lead to a concrete form of the core function.

\begin{table}[!htp]
\begin{center}
\begin{tabular}{|p{6.5cm}|p{9.2cm}|}
\hline
\tabincell{c}{Specific choices for $G(u)$}&
\tabincell{l}{(i) Gausse-Seidel Type\\ \hspace{0.5cm}$\sum_{i=1}^{N}G(u_1^{k+1},\cdots,u_{i-1}^{k+1},u_i,u_{i+1}^k,\cdots,u_N^k)$\\(ii) Jacobian Type\\ \hspace{0.5cm}$\sum_{i=1}^{N}G(u_1^{k},\cdots,u_{i-1}^{k},u_{i},u_{i+1}^{k},\cdots,u_{N}^{k})$}\\
\hline
\tabincell{l}{Specific choices for $\frac{\gamma}{2}\|\sum_{i=1}^NA_iu_i-b\|^2$} &
\tabincell{l}{(i) Gausse-Seidel Type\\ \hspace{0.5cm}$\frac{\gamma}{2}\sum_{i=1}^{N}\|A_iu_i+\sum\limits_{j<i}A_ju_j^{k+1}+\sum\limits_{j>i}A_ju_j^k-b\|^2$\\(ii) Jacobian Type\\ \hspace{0.5cm}$\frac{\gamma}{2}\sum_{i=1}^{N}\|A_iu_i+\sum_{j\neq i}A_ju_j^{k}-b\|^2$}\\
\hline
\tabincell{l}{Specific choices for $\frac{1}{2}\sum_i^{N} \| u_i\|^2_{M_i}$} &
\tabincell{l}{(i) Newton/quassi-Newton type\\ \hspace{0.4cm}$M_i$ is the block diagonal of (approximated) Hessian\\ \hspace{0.4cm}matrix of $G(u^k)$ \\(ii) Constant matrix Type\\ \hspace{0.4cm}$M_i$ could be zero matrix, identity matrix or $A_i^{\top}A_{i}$  }\\
\hline
\end{tabular}
\end{center}
\caption{some choices for the three parts included in core functions $K^k(\cdot)$\label{table1}}
\end{table}
Figure~\ref{fg1} describes the process of VAPP-AL to derive the specific algorithm for solving practical problems.
\begin{figure}[h]
\centering\resizebox{15cm}{8cm}{\includegraphics{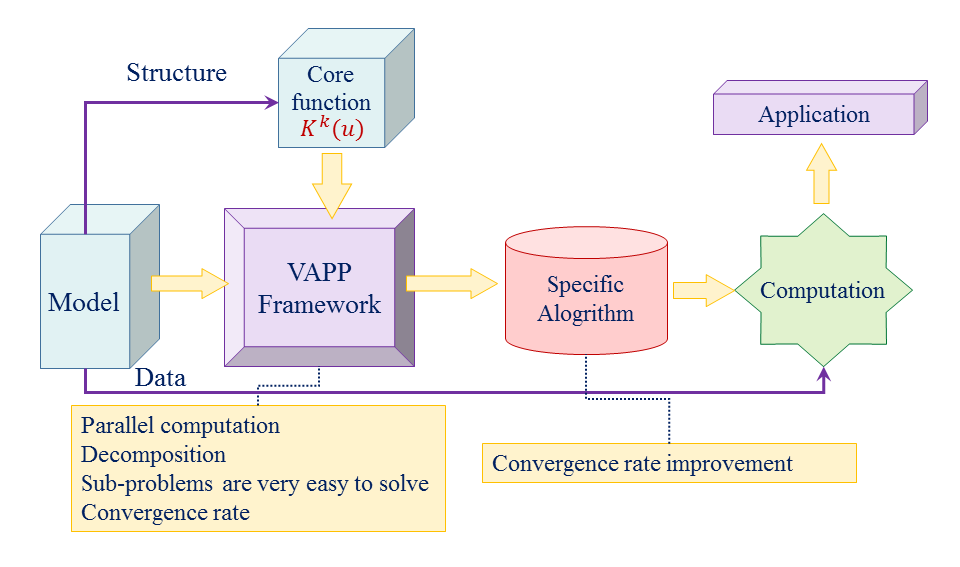}}
\caption{\label{fg1}Process of VAPP-AL to Derive the Specific Algorithm for Solving Practical Problems }
\end{figure}

Next we shall illustrate how to choose core functions in three practical problems. In particular, our core functions satisfy Assumption~\ref{assump2} and the resulting subproblems admit closed-form solutions.
Before introducing those problems, we first recall the so-called soft-shrinkage operator, which is formally defined as:\\
\begin{equation}
u^*=shrink(r,1/\mu)\triangleq sign(r)\cdot\max\{0,|r|-1/\mu\} \label{eq:shrink}
\end{equation}
where $\mu>0$, $r\in\RR^n$ and $sign(\cdot)$ is the sign function. It is well known \cite{Fusedlassosvm} that, $u^*$ is the optimal solution of the problem
\begin{eqnarray*}
\arg\min_{u\in\RR^n}\|u\|_1+\frac{\mu}{2}\|u-r\|^2.
\end{eqnarray*}
In the following, soft-shrinkage operator will be used to drive the expressions of solutions for subproblems of VAPP-AL.\\

{\bf Fused LASSO regularized SVM with quadratic hinge loss function}\\
Recently, the fused LASSO regularized SVM with quadratic hinge loss function has been successfully applied to a variety of applications; for example the diagnosis of disease \cite{Fusedlassosvm}.
Fused-SVM problem is given by:
\begin{equation}\label{Prob:LSVM1}
\begin{array}{lll}
\min\limits_{u\in \mathbf{R}^n}& \frac{1}{m}\sum_{i=1}^{m}\big{[}\max\{0,1-b_i(B_i^{\top}u)\}\big{]}^2+\lambda_1\|u\|_{1}+\lambda_2\sum_{j=2}^{n}|u_{j}-u_{j-1}|
\end{array}
\end{equation}
where $B=[B_1,B_2,\cdots,B_m]$ is a $n\times m$ matrix with $B_i$ being its $i$-th column and $b=[b_1,b_2,\cdots,b_m]^{\top}$ is an $m$-dimensional vector.
By introducing a matrix
$$D=\left(\begin{array}{cccccc}
               1&    -1&             0&\cdots&\cdots &0       \\
               0&     1&            -1&\cdots&\cdots &0       \\
          \vdots&\vdots&        \vdots&      &       &\vdots       \\
               0&     0&             0&\cdots&1      &-1
 \end{array}\right) \in\RR^{n-1\times n},$$
Fused-SVM can be equivalently transformed into the following nonlinear programming with equality constraints.
\begin{equation}\label{Prob:LSVM}
\begin{array}{lll}
\min\limits_{u\in \mathbf{R}^n, z\in\mathbf{R}^{n-1}}& \frac{1}{m}\sum_{i=1}^{m}\big{[}\max\{0,1-b_i(B_i^{\top}u)\}\big{]}^2+\lambda_1\|u\|_{1}+\lambda_2\|z\|_1\\
\rm {s.t}               & Du-z=0.
\end{array}
\end{equation}
Denoting $D_i$ to be the $i$-th column of matrix $D$,
we take core function as
\begin{equation*}
K^{k}(u,z)=\frac{\gamma}{2}\sum_{j=1}^{n}\|D_{j}u_{j}+\sum_{l\neq j}D_{l}u_{l}^k-z^k\|_2^2+\frac{\alpha_1}{2}\sum_{j=1}^{n} \|u_{j}-u_{j}^k\|_2^2+\frac{\gamma}{2}\|Du^k-z\|_2^2+\frac{\alpha_2}{2}\|z-z^k\|_2^2,
\end{equation*}
and $\epsilon=1$, then we obtain the Jacobian type decomposition method from VAPP-AL.\\
Now let's investigate the solvability of the subproblems of this method. First, the solution of the subproblem for $u_{j}$, $j=1,\cdots,n$ is given by
\begin{eqnarray*}
u_{j}^{k+1}&=&\arg\min_{u_{j}\in \RR} -\frac{2}{m}\sum_{i=1}^{m}\langle\max\{0,1-b_i(\sum_{l=1}^nB_{il}^{\top}u_l^{k})\},b_iB_{ij}^{\top}u_j\rangle+\langle p^k,D_{j}u_{j}\rangle+\lambda_1\|u_{j}\|_1\\
            &&+\frac{\gamma}{2}\|D_ju_j+\sum_{l\neq j}D_lu_l^k-z^k\|^2+\frac{\alpha_1}{2}\|u_{j}-u_{j}^k\|_2^2.\\
           &=&\arg\min_{u_j\in \RR} \|u_j\|_1+\frac{\mu_1}{2}\|u_j-r_1\|_2^2\\
           &=&shink(r_1,1/\mu_1)
\end{eqnarray*}
where $\mu_1=\frac{\gamma\|D_j\|^2+\alpha_1}{\lambda_1}$ and $r_1=\frac{\alpha_1u_j^k-\gamma(\sum_{l\neq j}D_lu_l^k-z^k)+\frac{2}{m}\sum_{i=1}^{m}B_{ij}b_i^{\top}\max\{0,1-b_i(\sum_{l=1}^nB_{il}^{\top}u_l^k)\}-D_j^{\top}p^k}{\gamma\|D_j\|_2^2+\alpha_1}$.\\
Then the problem can be solved as follows to update $z$.
\begin{eqnarray*}
z^{k+1}&=&\arg\min_{z\in \RR^{n-1}} \langle p^k,-z\rangle+\lambda_2\|z\|_1+\frac{\gamma}{2}\|Du^k-z\|_2^2+\frac{\alpha_2}{2}\|z-z^k\|_2^2\\
       &=&\arg\min_{z\in \RR^{n-1}} \|z\|_1+\frac{\gamma+\alpha_2}{2\lambda_2}\|z-\frac{\gamma Du^k+\alpha_2z^k+p^k}{\gamma+\alpha_2}\|_2^2\\
       &=&shink(\frac{\gamma Du^k+\alpha_2z^k+p^k}{\gamma+\alpha_2},\frac{\lambda_2}{\gamma+\alpha_2}).
\end{eqnarray*}
{\bf $l_1$-regularized logistic regression}\\
$l_1$-regularized logistic regression can be described as follows
\begin{equation}\label{Prob:RLR}
\begin{array}{lll}
\min\limits_{u\in \mathbf{R}^n}& \frac{1}{m}\sum_{i=1}^{m}\log (1+\exp(-b_i(B_i^{\top}u)))+\lambda\|u\|_1
\end{array}
\end{equation}
where $B=[B_1,B_2,\cdots,B_m]$ is a $n\times m$ matrix with $B_i$ being its $i$-th column and $b=[b_1,b_2,\cdots,b_m]^{\top}$ is an $m$-dimensional vector.
It has been proposed as a promising method for classification problem in machine learning \cite{linearSVM}.
To apply VAPP-AL, we introduce a new variable $z\in\RR^m$ as well as equality constraints $B^{\top}u=z$, and obtain the following equivalent problem:
\begin{equation}\label{Prob:RLR1}
\begin{array}{lll}
\min\limits_{u\in \mathbf{R}^n, z\in\mathbf{R}^m}& \frac{1}{m}\sum_{i=1}^{m}\log (1+\exp(-b_iz_i))+\lambda\|u\|_1\\
\rm {s.t}                                        & B^{\top}u-z=0.
\end{array}
\end{equation}
Then we adopt Newton type core function $K^{k}(u,z)=\frac{1}{2}\sum_{j=1}^{n} \|u_{j}-u_{j}^k\|_2^2+\frac{1}{2}\sum_{i=1}^{m}\|z_{i}-z_{i}^k\|_{M_{i}^k}^2$, where
$$M^k=\left(\begin{array}{ccccc}
         M_{1}^k&      &              &      &              \\
                &\ddots&              &      &              \\
                &      &M_{i}^k       &      &              \\
                &      &              &\ddots&              \\
                &      &              &      &M_{m}^k
 \end{array}\right)=\left(\begin{array}{ccccc}
  \frac{(b_1)^2\exp(-b_1v_1^k)}{m(1+\exp(-b_1z_1))^2}&      &              &      &              \\
                &\ddots&              &      &              \\
                &      &\frac{(b_i)^2\exp(-b_iv_i^k)}{m(1+\exp(-b_iz_i))^2}&      &              \\
                &      &              &\ddots&              \\
                &      &              &      &\frac{(b_m)^2\exp(-b_mv_m^k)}{m(1+\exp(-b_mz_m))^2}
 \end{array}\right)$$
is the Hessian matrix of loss function $f(z)=\frac{1}{m}\sum_{i=1}^{m}\log (1+\exp(-b_iz_i))$ with $M_{i}^k$ being the $i$-th component of the diagonal.
According to VAPP-AL algorithm, we can write out the 1-dimensional subproblems for $u_{j}$, $j=1,\cdots,n$ and $z_{i}$, $i=1,\cdots,m$, respectively.
Let $s_j$ be the $j$-th element of vector $s=B(p^k+\gamma(B^{\top}u^k-z^k))$, then
\begin{equation*}
u_{j}^{k+1}=\arg\min_{u_{j}\in \RR} \langle s_j, u_{j}\rangle+\lambda\|u_{j}\|_1+\frac{1}{2\epsilon}\|u_{j}-u_{j}^k\|_2^2.
\end{equation*}
By applying shrinkage operator, we obtain that
$u_j^{k+1}=shrink(u_j^{k}-\epsilon s_j,\epsilon\lambda)$. The subproblem of $z_i$ is given by
\begin{equation*}
z_{i}^{k+1}=\arg\min_{z_{i}\in \RR} \langle\frac{-b_i\exp(-b_iz_i^k)}{m(1+\exp(-b_iz_i^k))},z_i\rangle+\langle p_{i}^k+\gamma(B_{i}^{\top}u^k-z_{i}^k),-z_{i}\rangle+\frac{1}{2\epsilon}\|z_{i}-z_{i}^k\|_{M_i^k}^2,\; i=1,\cdots,m,
\end{equation*}
where $p_{i}$ is the $i$-th element of $m$-dimensional vector $p^k$. It is indeed an unconstrained quadratic problem and
admits a closed-form solution:
$z_i^{k+1}=z_i^k+\frac{\epsilon b_i\exp(-b_iz_i^k)}{M_{i}^km(1+\exp(-b_iz_i^k))}+\frac{\epsilon(p_{i}^k+\gamma(B_{i}^{\top}u^k-z_{i}^k))}{M_{i}^k}$.\\

{\bf Dual Support Vector Machine Problem}\\
Dual Support Vector Machine problem (DSVM) introduced in~\cite{Boser, Cortes, Platt} promotes a quadratic programming as follows:
\begin{equation}\label{Prob:DSVM}
\begin{array}{lll}
\min\limits_{u\in \mathbf{R}^n}& \frac{1}{2}u^{\top}Q u-e^{\top}u,\\
\rm {s.t}                      & y^{\top}u=0,\\
                               & 0\leq u\leq c.
\end{array}
\end{equation}
where $Q$ is a $n\times n$ p.s.d. symmetric matrix, $e$ and $y$ are two $n$-dimensional vectors.\\
We adopt core function as follows
\begin{eqnarray*}
K^{k}(u)&=&\frac{1}{2}\|u-u^k\|_2^2
\end{eqnarray*}
for VAPP-AL, then the iterative scheme of VAPP-AL leads to that
\begin{eqnarray*}
u^{k+1}&=&\arg\min_{0\leq u\leq c} \langle Qu^k,u\rangle - e^{\top}u + p^ky^{\top}u+\gamma y^{\top}u^ky^{\top}u+\frac{1}{2\epsilon}\|u-u^k\|_{2}^2\\
       &=&\min\{c,\max\{0,u^k-\epsilon(Qu^k+p^ky+\gamma y^{\top}u^ky-e)\}\},\\
p^{k+1}&=&p^k+\rho y^{\top}u^{k+1}.
\end{eqnarray*}

\section{Convergence Analysis for VAPP-AL}\label{sec:convergence}
In this section, we present the main convergence theorem of this paper, which states that the sequence generated by VAPP-AL is actually convergent under some mild conditions.
\begin{theorem}\label{theo:general-convergence} Suppose Assumption \ref{assump1} and Assumption \ref{assump2} hold. Then by picking
\begin{equation}\label{para-choice}0<\epsilon<\beta^{k}/\big{(}B_{G}+\gamma \cdot \lambda_{\max}(A^{\top}A)\big{)}\quad \mbox{and}\quad 0<\rho<(1 + \delta)\gamma,\quad\mbox{with some}\quad 0<\delta\leq1
\end{equation}
the sequence $\{(u^{k},p^{k})\}$ generated by VAPP-AL converges to $(u^{*},p^{*})$, which is the saddle point of $L$ over $U\times R^{m}$.
\end{theorem}
We would like to remark that the standard convergence analysis of ADMM
works only when the coupled term in the objective of problem~\eqref{Prob:general-function} disappears, i.e., $G(u)=0$.
However, when this is not the case, our algorithm still works and the convergence is guaranteed. The rest of this section is dedicated to the proof of Theorem~\ref{theo:general-convergence}. To this end, we need to present some lemmas as preparation. First of all, we list some useful facts of differentiable convex below.

\begin{lemma}\label{lemma:Lipschitz3point}
\noindent Let function $f$ be convex and differentiable on $U$.
\begin{itemize}
\item[(i)] If $f$ is strongly convex with constant $\beta_f$, then
\begin{equation}\label{LB}
\mbox{for all } u,v\in U, f(u)-f(v)\geq\langle\nabla f(v),u-v\rangle+\frac{\beta_f}{2}\|u-v\|^2.
\end{equation}
\item[(ii)] If the derivative of $f$ is Lipschitz with constant $B_f$, then
\begin{equation}\label{UB}
\mbox{for all } u,v\in U, f(u)-f(v)\leq \langle\nabla f(v),u-v\rangle+\frac{B_f}{2}\|u-v\|^2,
\end{equation}
    and
\begin{equation}\label{threePoint}
\mbox{for all } u,v\in U, f(v)-f(w)\leq \langle\nabla f(u),v-w\rangle+\frac{B_f}{2}\|u-v\|^2.
\end{equation}
\end{itemize}
\end{lemma}
\begin{proof}
\noindent The results~\eqref{LB} and~\eqref{UB} are classical. From the convexity of $f$ and~\eqref{UB}, we have
\begin{eqnarray*}
\langle\nabla f(u),w-v\rangle&=&\langle\nabla f(u),w-u\rangle+\langle\nabla f(u),u-v\rangle\leq [f(w)-f(u)]+[f(u)-f(v)+\frac{B_{f}}{2}\|u-v\|^{2}]\\
&=&f(w)-f(v)+\frac{B_{f}}{2}\|u-v\|^{2}.
\end{eqnarray*}
\end{proof}

Then we have following result regarding the sequence generated by VAPP-AL.
\begin{lemma}\label{lemma:KK-diff}
Suppose Assumption \ref{assump1} and Assumption \ref{assump2} hold, $(u^{*},p^{*})$ is any saddle point of $L$, and ${(u^{k},p^{k})}$ is generated by VAPP-AL. Then it holds that
\begin{eqnarray}
&&\langle \nabla \Kk(u^{k})-\nabla \Kk(u^{k+1}),u^{*}-u^{k+1}\rangle + \frac{\epsilon}{2\rho}\left(\|p^{k+1}-p^{*}\|^{2}-\|p^{k}-p^{*}\|^{2} \right) \nonumber\\
&\le& \frac{\epsilon}{2} \left({B_{G}}\|u^k - u^{k+1}\|^2+(\rho-\gamma)\|Au^{k+1}-b\|^2 -\gamma \|Au^{k}-b\|^{2}+\gamma \|A(u^{k+1}-u^{k})\|^{2} \right) \label{eq:KK-diff}
\end{eqnarray}
\end{lemma}
\begin{proof} Recall that in every iteration of VAPP the subproblem is given by
\begin{equation}\label{VAPP:subprob}
u^{k+1}=\arg\min\limits_{u\in U}K^{k}(u)- \langle \nabla K^{k}(u^{k}), u \rangle +\epsilon\,[ \langle \nabla G(u^{k}), u \rangle + J(u)
+ \langle p^{k}+\gamma(Au^{k}-b)), Au\rangle ]
\end{equation}
Since $K^{k}(\cdot)$ is strongly convex, the unique solution $u^{k+1}$ of is characterized by the following variational inequality:
\begin{eqnarray}
\langle \nabla \Kk(u^{k+1})-\nabla \Kk(u^{k}),u-u^{k+1}\rangle+ \epsilon\bigg{(} \langle  \nabla G(u^{k}),u-u^{k+1}\rangle +J(u)-J(u^{k+1}) \nonumber\\
+ \langle p^{k}+\gamma(Au^{k}-b),A(u-u^{k+1})\rangle \bigg{)}\geq 0 \quad \forall u\in U. \label{eq:vi}
\end{eqnarray}
By taking $u=u^{*}$, one has that
\begin{eqnarray}
&&\langle \nabla \Kk(u^{k})-\nabla \Kk(u^{k+1}),u^{*}-u^{k+1}\rangle\nonumber\\
&\leq&\epsilon\left(\langle \nabla G(u^{k}), u^{*}-u^{k+1}\rangle + J(u^{*})-J(u^{k+1})+\langle p^{k}+\gamma(Au^{k}-b),A(u^{*}-u^{k+1})\rangle\right)\nonumber \\
&\le & \epsilon\left( (G+J)(u^{*})-(G+J)(u^{k+1})+\frac{B_{G}}{2}\|u^{k}-u^{k+1}\|^{2} +\langle p^{k}+\gamma(Au^{k}-b),A(u^{*}-u^{k+1})\rangle \right)\nonumber \\ \label{eq:deltaKk}
\end{eqnarray}
where the second last inequality is due to Lemma \ref{lemma:Lipschitz3point}. Furthermore, since $(u^{*},p^{*})$ is a saddle point,
\begin{eqnarray}
(G+J)(u^{*}) &=& (G+J)(u^{*}) + \langle p^*, Au^*-b \rangle  \nonumber \\
&=& L(u^{*},p^{*}) \nonumber \\
&\le & L(u^{k+1},p^{*}) = (G+J)(u^{k+1}) + \langle p^*, Au^{k+1}-b \rangle
\label{eq:9}
\end{eqnarray}
Combining \eqref{eq:deltaKk}, \eqref{eq:9} and the fact $Au^* = b$ yields
\begin{eqnarray}
&&\langle \nabla \Kk(u^{k})-\nabla \Kk(u^{k+1}),u^{*}-u^{k+1}\rangle \nonumber \\
&\le& \epsilon\left(\frac{B_{G}}{2}\|u^{k}-u^{k+1}\|^{2} +\langle p^*-p^{k} - \gamma(Au^{k}-b),Au^{k+1}-b\rangle \right) \label{eq:deltaKk2}
\end{eqnarray}
On the other hand, from the dual update in VAPP-AL
\begin{equation}\label{VAPP:dual-update}
p^{k+1}=p^{k}+\rho(\Aukpmb),
\end{equation}
it follows that
\begin{eqnarray*}
&&\|p^{k+1}-p^{*}\|^{2}\\
&=&\|p^{k}-p^{*}\|^{2}+\rho^{2}\|Au^{k+1}-b\|^{2}+2\rho\langle p^{k}-p^{*},Au^{k+1}-b\rangle\\
&=&\|p^{k}-p^{*}\|^{2}+\rho(\rho-2\gamma )\|Au^{k+1}-b\|^{2}+2\rho\langle p^{k}-p^{*}+\gamma(Au^{k+1}-b),Au^{k+1}-b\rangle.
\end{eqnarray*}
As a result,
\begin{eqnarray}
&&\frac{\epsilon}{2\rho}\|p^{k+1}-p^{*}\|^{2}-\frac{\epsilon}{2\rho}\|p^{k}-p^{*}\|^{2}\nonumber\\
&=&\epsilon \left(\frac{(\rho-2\gamma )}{2}\|Au^{k+1}-b\|^{2}+ \langle p^{k}-p^{*}+\gamma(Au^{k+1}-b),Au^{k+1}-b\rangle \right).\label{eq:deltap}
\end{eqnarray}
Adding \eqref{eq:deltap} and \eqref{eq:deltaKk2} together yields that
\begin{eqnarray*}
&&\langle \nabla \Kk(u^{k})-\nabla \Kk(u^{k+1}),u^{*}-u^{k+1}\rangle + \frac{\epsilon}{2\rho}\left(\|p^{k+1}-p^{*}\|^{2}-\|p^{k}-p^{*}\|^{2} \right)\\
&\le& \frac{\epsilon}{2} \left({B_{G}}\|u^k - u^{k+1}\|^2+(\rho - \gamma)\|Au^{k+1}-b\|^2 - \gamma\|Au^{k+1} -b\|^2+2\gamma \langle A(u^{k+1}-u^{k}), Au^{k+1}-b \rangle \right) \\
&=&  \frac{\epsilon}{2} \left({B_{G}}\|u^k - u^{k+1}\|^2+(\rho - \gamma)\|Au^{k+1}-b\|^2 -\gamma \|Au^{k}-b\|^{2}+\gamma \|A(u^{k+1}-u^{k})\|^{2} \right).
\end{eqnarray*}
\end{proof}

\begin{lemma}\label{lemma:unit}
Suppose that there are positive $\alpha$, $T$ and the sequence of functions $\{\mu^k(u,v)\}$ satisfies the following conditions:
\begin{equation}
\frac{\alpha}{2}\|u-v\|^2\leq\mu^k(u,v)\leq\frac{T}{2}\|u-v\|^2, \forall u,v\in U.\label{eq:mu-Bound}
\end{equation}
Moreover, let the sequence $\{m^{k}\}$ be constructed such that
\begin{equation}
0<\delta\leq m^{k+1}\leq \frac{\alpha}{T}m^{k}\leq m^{k}\leq 1, \mbox{with some}\quad 0<\delta\leq1,\label{eq:mk}
\end{equation}
then
\begin{itemize}
\item[(i)] $\mu^k(u,v)$ is a Distance-like function on $U\times U$, i.e., $$\mbox{for any } u,v\in U, \mbox{ any } k\in\mathbb{N}, \mu^k(u,v)\geq 0 \mbox{ and } \mu^k(u,v)=0 \mbox{ iff } u=v.$$
\item[(ii)] for all $u,v\in U$, one has
\begin{equation}
m^{k+1}\mu^{k+1}(u,v)\leq m^{k}\mu^k(u,v).\label{eq:KK-vary}
\end{equation}
\end{itemize}
\end{lemma}
\begin{proof}
\noindent (i) is obviously from~\eqref{eq:mu-Bound}. For (ii), since sequence $\{m^{k}\}$ satisfies~\eqref{eq:mk}, then for all $u,v\in U$ we have
\begin{eqnarray*}
\mu^{k+1}(u,v)&\leq& \frac{T}{2}\|u-v\|^{2}=\frac{T}{\alpha}\frac{\alpha}{2}\|u-v\|^{2}\\
&\leq& \frac{T}{\alpha} \mu^{k}(u,v)\leq \frac{m^{k}}{m^{k+1}} \mu^{k}(u,v),
\end{eqnarray*}
and the desired result follows.
\end{proof}

Now we are ready to prove the convergence of VAPP-AL.

{\bf Proof of Theorem~\ref{theo:general-convergence}}
\begin{proof} According to Lemma~\ref{lemma:unit}, we can take the sequence $\{m^k\}$ such that
\begin{equation}\label{eq:mk-convergence}
0<\delta\leq m^{k+1}\leq \frac{\beta}{B}m^{k}\leq m^{k}\leq 1,\quad\mbox{with some}\quad 0<\delta\leq1.
\end{equation}
Moreover, for the definition of $D_K^k(u,v)$ and Assumption~\ref{assump2} of $\Kk$, we have
\begin{eqnarray*}
D_K^k(u,v)\geq\frac{\beta^k}{2}\|u-v\|^2\geq\frac{\beta}{2}\|u-v\|^2,\\
D_K^k(u,v)\leq\frac{B^k}{2}\|u-v\|^2\leq\frac{B}{2}\|u-v\|^2.
\end{eqnarray*}
Consequently, from Lemma~\ref{lemma:unit}, we obtain
\begin{equation}\label{equation_KKunite}
m^{k+1}D_K^{k+1}(u,v)\leq m^{k}D_K^k(u,v).
\end{equation}
Then for the pair $(u^*,p^*)$ saddle point of the Lagrangian of (P), we define the following function
\begin{eqnarray}\label{func:Lambda}
\Lambda^k(u,p)=m^k\left( D_K^k(u^*,u) +\frac{\epsilon}{2\rho}\|p-p^{*}\|^2 - \frac{\gamma \epsilon}{2} \|Au - b\|^2\right),
\end{eqnarray}
where $(u^*,p^*)$ is a saddle point associated with the Lagrangian function $L(u,p)$.
Then due to the strongly convexity of $\Kk$, we have that
\begin{eqnarray}
\Lambda^k(u,p)&=&m^k\left( D_K^k(u^*,u) +\frac{\epsilon}{2\rho}\|p-p^{*}\|^2 - \frac{\gamma \epsilon}{2} \|Au - b\|^2\right)\nonumber\\
&\geq&m^k\left( \frac{\beta^{k}}{2}\|u^*-u\|^{2} + \frac{\epsilon}{2\rho}\|p-p^{*}\|^2-\frac{\gamma \epsilon}{2}\lambda_{max}(A^{\top}A)\|u-u^{*}\|^{2}\right)  \nonumber \\
            &\geq& m^k\left( \frac{1}{2}\left(\beta^{k}-\gamma\epsilon  \lambda_{max}(A^{\top}A)\right)\|u-u^{*}\|^{2}+\frac{\epsilon}{2\rho}\|p-p^{*}\|^{2}\right)  \geq 0 \label{nonegative-measure-func},
\end{eqnarray}
where the last inequality follows from the choice of $\epsilon$ and $\rho$ in~\eqref{para-choice}. That is the distance between $(u,p)$ and saddle point $(u^*,p^*)$ is quantified by $\Lambda^k(u,p)$. As a result, to achieve the convergence, it suffices to study the sequence $\{\Lambda^k(u^{k},p^{k})\}$.

According to Assumption~\ref{assump2} and~\eqref{equation_KKunite}, one has that
\begin{eqnarray}
&&\Lambda^{k+1}(u^{k+1},p^{k+1})-\Lambda^k(u^{k},p^{k}) \nonumber \\
& = & m^{k+1}D_K^{k+1}(u^*,u^{k+1})-m^k D_K^k(u^*,u^k)+\frac{m^{k+1}\epsilon}{2\rho} \|p^{k+1}-p^{*}\|^{2}- \frac{m^{k+1}\gamma \epsilon}{2} \|Au^{k+1} - b\|^2  \nonumber \\
&&    - \frac{m^{k}\epsilon}{2\rho} \|p^{k}-p^{*}\|^{2}  + \frac{m^k\gamma \epsilon}{2} \|Au^k - b\|^2\nonumber \\
&\le& m^{k}\left(\Kk(u^{*}) -K^{k}(u^{k+1})-\langle \nabla K^{k}(u^{k+1}),u^{*}-u^{k+1}\rangle - \Kk(u^{*}) +  K^{k}(u^{k}) + \langle \nabla \Kk(u^{k}),u^{*}-u^{k}\rangle  \right) \nonumber \\
&& +\frac{m^{k}\epsilon}{2\rho} \|p^{k+1}-p^{*}\|^{2}- \frac{m^{k}\epsilon}{2\rho} \|p^{k}-p^{*}\|^{2} - \frac{m^{k+1}\gamma \epsilon}{2} \|Au^{k+1} - b\|^2 + \frac{m^k\gamma \epsilon}{2} \|Au^k - b\|^2 \nonumber \\
&=& m^k\left( \Kk(u^{k})-\Kk(u^{k+1})- \langle \nabla \Kk(u^{k}),u^{k}-u^{k+1}\rangle  \right) \label{first-term} \\
&& +m^k\left( \langle\nabla \Kk(u^{k}) - \nabla \Kk(u^{k+1}),u^{*}-u^{k+1}\rangle + \frac{\epsilon}{2\rho} \big{(}\|p^{k+1}-p^{*}\|^{2}-  \|p^{k}-p^{*}\|^{2} \big{)}\right) \label{second-term}\\
&& - \frac{m^{k+1}\gamma \epsilon}{2} \|Au^{k+1} - b\|^2 + \frac{m^k\gamma \epsilon}{2} \|Au^k - b\|^2. \nonumber
\end{eqnarray}
The convexity of $\Kk(\cdot)$ implies that formula~\eqref{first-term} is less than $-\frac{m^k\beta^{k}}{2}\|u^{k}-u^{k+1}\|^{2}$. Note that \eqref{second-term}
can further be bounded above by using Lemma~\ref{lemma:KK-diff}. Therefore,
\begin{eqnarray}
    && \Lambda^{k+1}(u^{k+1},p^{k+1})-\Lambda^k(u^{k},p^{k})\nonumber\\
    &\leq& m^k\left(  -\frac{(\beta^{k}-B_{G}\epsilon)}{2}\|u^{k}-u^{k+1}\|^{2} + \frac{\epsilon}{2} \left( (\rho-(1+\frac{m^{k+1}}{m^k})\gamma)\|Au^{k+1}-b\|^2 +\gamma \|A(u^{k+1}-u^{k})\|^{2} \right) \right) \nonumber\\
     &\leq& m^k\left(  -\frac{(\beta^{k}-B_{G}\epsilon)}{2}\|u^{k}-u^{k+1}\|^{2} + \frac{\gamma\, \epsilon}{2}\lambda_{max}(A^{\top}A)\|u^{k+1}-u^{k}\|^{2} + \frac{\epsilon}{2} (\rho-(1 + \delta)\gamma)\|Au^{k+1}-b\|^2 \right)  \nonumber\\
&\leq& m^k\left( \frac{1}{2}\left(\epsilon\left(B_{G}+\gamma \lambda_{max}(A^{\top}A)\right)-\beta^{k}\right)\|u^{k}-u^{k+1}\|^{2} +\frac{\epsilon}{2}(\rho-(1 + \delta)\gamma )\|Au^{k+1}-b\|^{2}\right)  \le 0.\label{up-bound-Lambda}
\end{eqnarray}
where the second inequality follows as $\delta\le m^{k+1}/m^k  $ and the last inequality is due to the choice of $\epsilon$ and $\rho$ defined in~\eqref{para-choice}. Consequently $\{\Lambda^k(u^{k},p^{k})\}$ is nonincreasing. This combined with~\eqref{nonegative-measure-func} implies that
 $\{\Lambda^k(u^{k},p^{k})\}$ has a limit,
\begin{equation}
\lim_{k\to \infty}\|p^{k+1} - p^{k}\|/\rho = \lim_{k\to \infty}\|Au^{k+1}-b\|=0 \quad \mbox{and}\quad \lim_{k\to \infty}\|u^{k+1}-u^{k}\|=0\label{eq:10}
\end{equation}
Moreover, \eqref{nonegative-measure-func} and boundedness of $\{\Lambda^k(u^{k},p^{k})\}$ implies that $\{u^{k}\}$ and $\{p^{k}\}$ are bounded as well.
Therefore the sequence $\{(u^{k},p^{k})\}$ has a cluster point $(\bar{u},\bar{p})$. Taking the limit in~\eqref{eq:10} gives that
\begin{equation}\label{bar-u-feasible}
\quad\hspace{5cm} A\bar{u}-b=0.
\end{equation}
Furthermore, since gradients of $\Kk$ and $G$ are Lipschitz continuous,
$$
\lim_{k\to \infty}\|\nabla \Kk(u^{k+1})-\nabla \Kk(u^{k})\|=0 \quad \mbox{and}\quad \lim_{k\to \infty}\|\nabla G(u^{k+1})-\nabla G(u^{k})\|=0.
$$
Now letting $k+1 \to \infty$ in~\eqref{eq:vi} and then combining the formulas above together with the convexity of $G$ and~\eqref{eq:10} yields
\begin{equation}\label{eq:11}
\quad\hspace{2cm}(G+J)(u)-(G+J)(\bar{u})+\langle \bar{p},Au-b\rangle \geq 0,\forall u\in U.
\end{equation}
Consequently \eqref{bar-u-feasible} and \eqref{eq:11} imply that
\begin{equation*}
L(\bar{u},p)=L(\bar{u},\bar{p})=(G+J)(\bar{u})+\langle p, A\bar{u}-b\rangle\leq (G+J)(u)+\langle \bar{p},Au-b\rangle=L(u,\bar{p}),\;\forall p\in R^{m},\;\forall u\in U.
\end{equation*}
From the definition of saddle point, it holds that $(\bar{u},\bar{p})\in U^{*}\times P^{*}$.

Note that the argument above goes through as long as $(u^{*},p^{*})$ is a saddle point of Lagrangian function $L(u , p)$.
Therefore, we can set $u^{*}=\bar{u}$, $p^{*}=\bar{p}$ and taking limit for sequence $\{\Lambda^k(u^{k},p^{k})\}$. From construction~\eqref{func:Lambda} of $\Lambda^k(u,v)$, we know that
zero is a cluster point of $\{\Lambda^k(u^{k},p^{k})\}$.
Moreover, we have shown that the limit of the total sequence $\{\Lambda^k(u^{k},p^{k})\}$ exists. As a result, $\Lambda^k(u^{k},p^{k}) \to 0$. This combined with \eqref{nonegative-measure-func} implies that  $u^{k}\rightarrow u^{*}$ and $p^{k}\rightarrow p^{*}$.
\end{proof}

\section{Complexity analysis in Ergodic and non-Ergodic sense}
\subsection{Convergence Rate Analysis in Ergodic sense}
We first analyze the rate of convergence in ergodic sense. From the variational inequalities system~\eqref{VIS_1}-\eqref{VIS_2} for (P), we can look at an equivalent variational inequality
\begin{equation}\label{VI}
H(w)-H(w^*)+\langle F(w), \Omega(w)-\Omega(w^*)\rangle\geq 0, \forall w\in\mathbf{W},
\end{equation}
where
\begin{eqnarray*}
w=\left(\begin{array}{c}u\\p\end{array}\right), H(w)=(G+J)(u), F(w)=\left(\begin{array}{c}p\\b-Au\end{array}\right), \Omega(w)=\left(\begin{array}{c}Au-b\\p\end{array}\right)\, \mbox{and}\,\, \mathbf{W}=U\times \mathbf{C}^*.
\end{eqnarray*}
In the rest, we show that after running $t$ iterations of the VAPP-AL, we can find a $w\in \mathbf{W}$ such that~\eqref{VI} is approximately satisfied with an error of $O(1/t)$, thus proving convergence rate of $O(1/t)$ for the VAPP-AL algorithm.\\
\begin{theorem}\label{thm:ergodic_iteration_complexity}
\noindent Suppose Assumptions of Theorem~\ref{theo:general-convergence} hold, then
\begin{itemize}
\item[(i)] we can construct the sequence $\{m^k\}$ satisfying
\begin{equation}\label{eq:mk-rate}
0<\delta\leq m^{k+1}\leq \frac{\beta-\epsilon B_G-\epsilon\gamma\lambda_{max}(A^{\top}A)}{B}m^{k}\leq m^{k}\leq 1,\quad\mbox{with some}\quad 0<\delta\leq1.
\end{equation}
\item[(ii)] Let $\{\tilde{w}^{k}\}$ be the sequence generated by the VAPP-AL:
\begin{eqnarray}
\tilde{w}^k=\left(\begin{array}{c}\tilde{u}^k\\ \tilde{p}^k\end{array}\right)=\left(\begin{array}{c}u^{k+1}\\ p^k+\gamma(Au^{k+1}-b)\end{array}\right).
\end{eqnarray}
\end{itemize}
For any integer number $t>0$, define $\bar{w}_{t}$ as $\bar{w}_{t}=\frac{1}{\sigma^t}\sum_{k=0}^{t}m^k\tilde{w}^k$.
with $\sigma^t=\sum_{k=0}^{t}m^{k}$,\\
Then it holds that
\begin{eqnarray}
H(\bar{w}_{t})-H(w)+\langle F(w), \Omega(\bar{w}_{t})-\Omega(w)\rangle\leq \frac{1}{(t+1)\delta}[\frac{B}{2\epsilon}\|u-u^{0}\|^{2}+\frac{1}{2\rho}\|p-p^{0}\|^{2}],
\end{eqnarray}
i.e., $\bar{w}_{t}$ is an approximate saddle point of $L$ with the accuracy of $O(1/t)$
\end{theorem}
\begin{proof}
From the convexity of $(G+J)(u)$, for every $u\in U, p\in R^m$, we have
\begin{eqnarray}\label{eq:VI_rate}
&&H(\bar{w}_{t})-H(w)+\langle F(w), \Omega(\bar{w}_{t})-\Omega(w)\rangle\nonumber\\
&\leq&\frac{1}{\sigma^t}\sum_{k=0}^{t}m^{k}(G+J)(\tilde{u}^k)-(G+J)(u)+\frac{1}{\sigma^t}\sum_{k=0}^{t}m^{k}\langle p,A\tilde{u}^k-b\rangle-\langle p,Au-b\rangle+\langle b-Au,\frac{1}{\sigma^t}\sum_{k=0}^{t}m^{k}\tilde{p}^k-p\rangle\nonumber\\
&=&\frac{1}{\sigma^t}\sum_{k=0}^{t}m^{k}\bigg{[}H(\tilde{w}^{k})-H(w)+\langle F(w), \Omega(\tilde{w}^{k})-\Omega(w)\rangle\bigg{]}.
\end{eqnarray}
To get the upper bound of the variational term in the square brackets of~\eqref{eq:VI_rate}, Let $f^k(\cdot)=\frac{1}{\epsilon}K^k(\cdot)-G(\cdot)-\frac{\gamma}{2}\|Au\|^2$, construct the follows merit function
\begin{eqnarray*}
\Phi^{k}(u,v)&=&f^k(u)-f^k(v)-\langle\nabla f^k(v),u-v\rangle, \forall u,v\in U.
\end{eqnarray*}
From Assumptions of Theorem~\ref{theo:general-convergence}, we obtain the upper bound and lower bound for $\Phi^{k}(u,v)$ as follows:
\begin{eqnarray*}
&&\Phi^{k}(u,v)\leq\frac{B^k}{2\epsilon}\|u-v\|^2\leq\frac{B}{2\epsilon}\|u-v\|^2, \;\forall u,v\in U,\\
&&\Phi^{k}(u,v)\geq\frac{\beta^k-\epsilon B_G-\epsilon\lambda_{max}(A^{\top}A)}{2\epsilon}\|u-v\|^2\geq\frac{\beta-\epsilon B_G-\epsilon\lambda_{max}(A^{\top}A)}{2\epsilon}\|u-v\|^2\geq 0, \;\forall u,v\in U.
\end{eqnarray*}
Now we ready to estimate the upper bound of the variation term of the square brackets of~\eqref{eq:VI_rate}.
\begin{eqnarray}\label{eq:VI_bound0}
 &&H(\tilde{w}^{k})-H(w)+\langle F(w), \Omega(\tilde{w}^{k})-\Omega(w)\rangle\nonumber\\
&\leq&\langle\nabla G(\tilde{u}^{k}), \tilde{u}^{k}-u\rangle+J(\tilde{u}^{k})-J(u)+\langle p,A(\tilde{u}^{k}-u)\rangle+\langle b-Au, \tilde{p}^k-p\rangle\nonumber\\
&=&\langle\nabla G(\tilde{u}^{k}), \tilde{u}^{k}-u\rangle+J(\tilde{u}^{k})-J(u)+\langle \tilde{p}^{k},A(\tilde{u}^{k}-u)\rangle+\langle b-A\tilde{u}^{k}, \tilde{p}^k-p\rangle\nonumber\\
&&+\langle p-\tilde{p}^k, A(\tilde{u}^k-u)\rangle+\langle A(\tilde{u}^k-u), \tilde{p}^k-p\rangle\nonumber\\
&=&\langle\nabla G(\tilde{u}^{k}), \tilde{u}^{k}-u\rangle+J(\tilde{u}^{k})-J(u)+\langle \tilde{p}^{k},A(\tilde{u}^{k}-u)\rangle+\langle b-A\tilde{u}^{k}, \tilde{p}^k-p\rangle.
\end{eqnarray}
To estimate the first three terms of right hand side of~\eqref{eq:VI_bound0}, using variational inequality~\eqref{eq:vi} and $\tilde{u}^k=u^{k+1}$, $\tilde{p}^k=p^k+\gamma(Au^{k+1}-b)$, we have
\begin{eqnarray}\label{eq:VI1}
&&\langle\nabla G(\tilde{u}^{k}), \tilde{u}^{k}-u\rangle+J(\tilde{u}^{k})-J(u)+\langle \tilde{p}^{k},A(\tilde{u}^{k}-u)\rangle\nonumber\\
&\leq&-\frac{1}{\epsilon}\langle\nabla K^k(u^{k})-(K^k)'(u^{k+1}),u-u^{k+1}\rangle+\langle\nabla G(u^{k})-\nabla G(u^{k+1}),u-u^{k+1}\rangle\nonumber\\
&&+\gamma\langle A(u^k-u^{k+1}),A(u-u^{k+1})\rangle
\end{eqnarray}
Then using three point identity Lemma 3.1 of~\cite{ChenTeboulle93}, we obtain
\begin{eqnarray}\label{eq:VI2}
  &&-\frac{1}{\epsilon}\langle\nabla K^k(u^{k})-\nabla K^k(u^{k+1}),u-u^{k+1}\rangle+\langle\nabla G(u^{k})-\nabla G(u^{k+1}),u-u^{k+1}\rangle\nonumber\\
&&+\gamma\langle A(u^k-u^{k+1}),A(u-u^{k+1})\rangle\nonumber\\
&=&\Phi^{k}(u,u^{k})-\Phi^{k}(u,u^{k+1})-\Phi^{k}(u^{k+1},u^{k})\nonumber\\
&\leq&\Phi^{k}(u,u^{k})-\Phi^{k}(u,u^{k+1}).
\end{eqnarray}
As a result, we have
\begin{eqnarray}\label{eq:VI_solution1}
  &&\langle\nabla G(\tilde{u}^{k}), \tilde{u}^{k}-u\rangle+J(\tilde{u}^{k})-J(u)+\langle \tilde{p}^{k},A(\tilde{u}^{k}-u)\rangle\nonumber\\
&\leq&\Phi^{k}(u,u^{k})-\Phi^{k}(u,u^{k+1}).
\end{eqnarray}
For dual hand, by condition~\eqref{para-choice} for $\rho$, we estimate the last term of right hand side of~\eqref{eq:VI_bound0}:
\begin{eqnarray}\label{eq:VI_solution2}
\langle b-A\tilde{u}^{k}, \tilde{p}^k-p\rangle&=& \frac{1}{\rho}\langle p^k-p^{k+1},  \tilde{p}^k-p\rangle\nonumber\\
&=&\frac{1}{2\rho}[\|p-p^{k}\|^2-\|p-p^{k+1}\|^2]+\frac{1}{2\rho}[\|p^{k+1}-\tilde{p}^k\|^2-\|p^k-\tilde{p}^k\|^2]\nonumber\\
&=&\frac{1}{2\rho}[\|p-p^{k}\|^2-\|p-p^{k+1}\|^2+\|p^{k}+\rho(Au^{k+1}-b)-p^k-\gamma(Au^{k+1}-b)\|^2\nonumber\\
  &&-\|p^k-p^k-\gamma(Au^{k+1}-b)\|^2]\nonumber\\
&=&\frac{1}{2\rho}[\|p-p^{k}\|^2-\|p-p^{k+1}\|^2+(\rho-\gamma)^2\|Au^{k+1}-b\|^2-\gamma^2\|Au^{k+1}-b\|^2]\nonumber\\
&=&\frac{1}{2\rho}[\|p-p^{k}\|^2-\|p-p^{k+1}\|^2+\rho(\rho-2\gamma)\|Au^{k+1}-b\|^2]\nonumber\\
&\leq&\frac{1}{2\rho}[\|p-p^{k}\|^2-\|p-p^{k+1}\|^2]
\end{eqnarray}
Together~\eqref{eq:VI_solution1} and~\eqref{eq:VI_solution2}, we have
\begin{eqnarray}\label{eq:VI_solution}
  &&\langle\nabla G(\tilde{u}^{k}), \tilde{u}^{k}-u\rangle+J(\tilde{u}^{k})-J(u)+\langle \tilde{p}^{k},A(\tilde{u}^{k}-u)\rangle+\langle b-A\tilde{u}^{k}, \tilde{p}^k-p\rangle\nonumber\\
&\leq&\Phi^{k}(u,u^{k})-\Phi^{k}(u,u^{k+1})+\frac{1}{2\rho}[\|p-p^{k}\|^2-\|p-p^{k+1}\|^2].
\end{eqnarray}
Since parameter $m^k$ satisfies condition~\eqref{eq:mk-rate}, by Lemma~\ref{lemma:unit} with $\mu^k(u,v)=\Phi^{k}(u,v)$, we have
\begin{eqnarray*}
m^{k+1}\Phi^{k+1}(u,v)\leq m^{k}\Phi^{k}(u,v),
\end{eqnarray*}
which follows the fact
\begin{eqnarray}\label{eq:eachiteration}
&&m^k[H(\tilde{w}^{k})-H(w)+\langle F(w), \Omega(\tilde{w}^{k})-\Omega(w)\rangle]\nonumber\\
&\leq& m^{k}[\Phi^k(u,u^k)-\Phi^k(u,u^{k+1})+\frac{1}{2\rho}\|p-p^k\|^2-\frac{1}{2\rho}\|p-p^{k+1}\|^2]\nonumber\\
&\leq& m^{k}[\Phi^k(u,u^{k})+\frac{1}{2\rho}\|p-p^{k}\|^2]-m^{k+1}[\Phi^{k+1}(u,u^{k+1})+\frac{1}{2\rho}\|p-p^{k+1}\|^2].
\end{eqnarray}
Summing the inequality~\eqref{eq:eachiteration} over $k=0,1,\ldots,n$, we obtain
\begin{eqnarray*}
  &&\frac{1}{\sigma^t}\sum_{k=0}^{t}m^{k}\bigg{[}H(\tilde{w}^{k})-H(w)+\langle F(w), \Omega(\tilde{w}^{k})-\Omega(w)\rangle\bigg{]}\\
&\leq& \frac{m^0\bigg{[}\Phi^k(u,u^{0})+\frac{1}{2\rho}\|p-p^{0}\|^2\bigg{]}}{\sigma^t}
\end{eqnarray*}
Together with $\sigma^t\geq (t+1)\delta$ and~\eqref{eq:VI_rate}, we observe that
\begin{eqnarray*}
&&H(\bar{w}_{t})-H(w)+\langle F(w), \Omega(\bar{w}_{t})-\Omega(w)\rangle\\
&\leq& \frac{1}{(t+1)\delta}\bigg{[}\frac{B}{2\epsilon}\|u-u^{0}\|^2+\frac{1}{2\rho}\|p-p^{0}\|^2\bigg{]},
\end{eqnarray*}
and this complete the proof.
\end{proof}
Observe that Theorem~\ref{thm:ergodic_iteration_complexity} prompts VAPP-AL has the convergence rate $O(1/t)$ in the worst case.
\subsection{Convergence Rate Analysis in non-Ergodic sense}
In this subsection, we analyze the iteration complexity of VAPP-AL in non-ergodic sense. Based on~\eqref{eq:10}, we know that $\lim_{k\to \infty}\|u^{k+1}-u^{k}\|=0$ is a necessary condition for the convergence of VAPP-AL. Thus, $\|u^{k+1}-u^{k}\|^2+\|p^{k+1}-p^{k}\|$ can be used as a quantity to measure
the degree of convergence of the sequence $\{u^k,p^k\}$ to a critical point. The following theorem provides the convergence rate on $\|u^{k+1}-u^{k}\|^2+\|p^{k+1}-p^{k}\|$.
\begin{theorem}\label{theo:convergence-rate} Let the assumptions of Theorem \ref{theo:general-convergence} be satisfied, then there exists a small positive $\nu$ such that
\begin{equation*}
\min_{0\leq k\leq t}\|u^{k+1}-u^{k}\|^{2}+\|p^{k+1}-p^{k}\|^{2}\leq \frac{\Lambda^0(u^{0},p^{0})}{(t+1)\nu}.
\end{equation*}
where the sequence $\{(u^k,p^k)\}$ is generated by VAPP-AL, $\Lambda^{0}(u^0,p^0)$ is defined by~\eqref{func:Lambda} with $k=0$, $(u,p)=(u^0,p^0)$.
\end{theorem}
\begin{proof}
\noindent From inequality~\eqref{up-bound-Lambda}, we have
\begin{eqnarray*}
    && \Lambda^{k+1}(u^{k+1},p^{k+1})-\Lambda^k(u^{k},p^{k})\nonumber\\
&\leq&m_k\left( \frac{1}{2}\left(\epsilon\left(B_{G}+\gamma \lambda_{max}(A^{\top}A)\right)-\beta^{k}\right)\|u^{k}-u^{k+1}\|^{2} +\frac{\epsilon}{2}(\rho-(1+\delta)\gamma )\|Au^{k+1}-b\|^{2}\right)\\
&\leq&m_k\left( \frac{1}{2}\left(\epsilon\left(B_{G}+\gamma \lambda_{max}(A^{\top}A)\right)-\beta^{k}\right)\|u^{k}-u^{k+1}\|^{2} +\frac{\epsilon}{2\rho^2}(\rho-(1+\delta)\gamma )\|p^{k}-p^{k+1}\|^{2}\right)
\end{eqnarray*}
and $\lim\limits_{k \to \infty} \Lambda^k(u^{k},p^{k}) = 0$. Since $\frac{1}{2}\left(\epsilon\left(B_{G}+\gamma \lambda_{max}(A^{\top}A)\right)-\beta^{k}\right) < 0$ and $\frac{\epsilon}{2\rho^2}(\rho-(1+\delta)\gamma )< 0$, there exist positive numbers $\nu_1$ and $\nu_2$ such that
$$
\frac{m^k}{2}\left(\epsilon(B_{G}+\gamma\lambda_{max}(A^{\top}A))-\beta^k \right) \le - \nu_1 \quad\mbox{and}\quad\frac{m^k\epsilon(\rho-(1+\delta)\gamma )}{2\rho^2} \le - \nu_2 .
$$
By letting $\nu = \min\{\nu_1, \nu_2\}$, we have that
\begin{equation}\label{eq:sum1}
\Lambda^{k+1}(u^{k+1},p^{k+1})-\Lambda^k(u^{k},p^{k}) \le -\nu\left( \|u^{k}-u^{k+1} \|^{2}+\|p^{k}-p^{k+1}\|^{2} \right)
\end{equation}
Summing (\ref{eq:sum1}) over $k$ and taking limit yields
\begin{eqnarray*}
\sum_{k=0}^{t}\nu\left( \|u^{k}-u^{k+1}\|^{2}+\|p^{k}-p^{k+1}\|^{2} \right)\le \Lambda^0(u^{0},p^{0}).
\end{eqnarray*}
\noindent Moreover, the above inequity implies that
\begin{equation*}
\min_{0\leq k\leq t}[\|u^{k+1}-u^{k}\|^{2}+\|p^{k+1}-p^{k}\|^{2}]\leq \frac{\Lambda^0(u^{0},p^{0})}{(t+1)\nu}.
\end{equation*}
\end{proof}
\section{Convergence Rate Analysis with Jacobian Regularization Quadratic Core Function}\label{sec:rate-quadratic-core}
In this section, the core function is endowed with the quadratic form:
\begin{equation}\label{eq:ProximalJacobicore}
K^{k}(u)=\sum_{i=1}^{N}\frac{\theta_{i}}{2}\|A_{i}u_{i}+(\sum_{j\neq i}A_{j}u_{j}^{k}-b)\|^{2}+ \sum_{i=1}^N \frac{\alpha_i}{2}\|P_i u_i\|^2,
\end{equation}
where $\theta_i \ge 0$ and $\alpha_i \ge 0$ for all $i=1,\dots, N$. We shall show how to improve convergence rate when the core function has the above quadratic form.
To facilitate our discussion, we assume $G(\cdot)$ is twice differentiable and let
\begin{eqnarray*}
H:=\frac{1}{\epsilon}\left(\begin{array}{ccccc}
  \theta_1 A_{1}^{T}A_{1} + \alpha_1 P_{1}^{T}P_{1}&      &              &      &              \\
                &\ddots&              &      &              \\
                &      &\theta_i A_{i}^{T}A_{i} + \alpha_i P_{i}^{T}P_{i}&      &              \\
                &      &              &\ddots&              \\
                &      &              &      &\theta_N A_{N}^{T}A_{N} + \alpha_N P_{N}^{T}P_{N}
 \end{array}
  \right)
\end{eqnarray*}
\begin{eqnarray*}
Q^{k}:=\int_{0}^{1}\nabla^2G(u^{k}+\tau(u^{k-1}-u^{k}))d\tau\quad\mbox{and}\quad \tilde{H}^{k}:=H-\gamma A^{T}A-Q^{k}
\end{eqnarray*}
From Assumption~\ref{assump1}, we have that the hessian matrix of $G(\cdot)$ is bounded above, i.e., there exists some $0 < B_G < \infty$ such that
$0 \preceq \nabla^2G(u) \preceq B_G I,\;\forall \; u \in U$. Consequently, it holds that
$0  \preceq Q^{k} \preceq B_G I$. We denote
\begin{equation}\label{bounds-notation}
\underline{H}: = H-\gamma A^{T}A - B_G I,\quad \overline{H}: = H-\gamma A^{T}A.
\end{equation}
Then, obviously
$\underline{H} \preceq \tilde{H}^{k} \preceq \overline{H} $ for all $k$.

The following proposition provides convergence of VAPP-AL with Jacobian regularization quadratic core function, whose condition is weaker than that of Theorem~\ref{theo:general-convergence}.
\begin{proposition}\label{prop:choice-core} When the core function is given by~\eqref{eq:ProximalJacobicore}, Assumption~\ref{assump2} is automatically satisfied if for each index $i=1,\dots, N$, we either have $A_{i}$ has full column rank with $\theta_i > 0$ or $P_{i}$ has full column rank with $\alpha_i >0$.\\
(i) Moreover, suppose $\underline{H}\succeq 0$ and $\rho\leq (1+\delta)\gamma$, with some $0<\delta\leq 1$ then the convergence of VAPP-AL is guaranteed.\\
(ii) In addition, the convergent condition $0<\epsilon<\beta^k/(B_G+\gamma\cdot\lambda_{\max}(A^{\top}A))$ in Theorem~\ref{theo:general-convergence} implies that $\underline{H}\succeq 0$. That is to say VAPP-AL actually converges under a weaker condition with Jacobian regularization quadratic core function.

\end{proposition}
\begin{proof} (i) For $K^k(u)$ is given by~\eqref{eq:ProximalJacobicore} and $\Lambda^{k}(u,p)$ is defined in~\eqref{func:Lambda}, from~\eqref{first-term} and~\eqref{second-term}, we have that
\begin{eqnarray}
    &&\Lambda^{k+1}(u^{k+1},p^{k+1})-\Lambda^k(u^{k},p^{k})\\
&\leq& m^k\big{(}  -\frac{1}{2}\sum_{i=1}^{N}\|(\theta_i A_{i}+\alpha_i P_{i})(u_{i}^{k}-u_{i}^{k+1})\|^{2}+\frac{B_{G}\epsilon}{2}\|u^{k}-u^{k+1}\|^{2} + \frac{\epsilon}{2}\gamma \|A(u^{k+1}-u^{k})\|^{2}\nonumber\\
&&+ \frac{\epsilon}{2} (\rho-(1+\frac{m^k}{m^{k+1}})\gamma)\|Au^{k+1}-b\|^2  \big{)} \nonumber\\
&=&m^k\left( -\frac{\epsilon}{2}\|u^{k}-u^{k+1}\|_{\underline{H}}^{2} +\frac{\epsilon}{2}(\rho-(1+\delta)\gamma )\|Au^{k+1}-b\|^{2}\right).\label{eq:quadraticconvergence}
\end{eqnarray}
For $\underline{H}\succeq 0$ and $\rho\leq (1+\delta)\gamma$, we derive from~\eqref{eq:quadraticconvergence} that
\begin{eqnarray*}
\Lambda^{k+1}(u^{k+1},p^{k+1})-\Lambda^k(u^{k},p^{k})\leq 0.
\end{eqnarray*}
The following convergence proof is similar to the proof of Theorem~\ref{theo:general-convergence}.
Therefore, the convergence is hold.\\
(ii) It is obviously that, when the core function is given by~\eqref{eq:ProximalJacobicore}, we have
$$\beta^{k}=\min_{i}\theta_{i}\lambda_{min}(A_{i}^{\top}A_{i})+\min_{i}\alpha_{i}\lambda_{min}(P_{i}^{\top}P_{i})$$
when the core function is given in~\eqref{eq:ProximalJacobicore}. Due to the definition of $H$, we have that
$$\lambda_{min}(H) = \frac{1}{\epsilon}\left(\min_{i}\theta_{i}\lambda_{min}(A_{i}^{\top}A_{i})+\min_{i}\alpha_{i}\lambda_{min}(P_{i}^{\top}P_{i})\right) = \frac{\beta_k}{\epsilon}$$
Consequently,
\begin{eqnarray*}
\langle \underline{H}x,x\rangle &=& \langle (H-\gamma A^{T}A-B_G I)x,x\rangle\\
                            &\geq& \left(\lambda_{min}(H)-\gamma\lambda_{max}(A^{\top}A)-B_G I\right)\|x\|^{2}\\
                            &=& \left(\frac{\beta^{k}}{\epsilon}-\gamma\lambda_{max}(A^{\top}A-B_G)\right)\|x\|^{2}\geq 0,
\end{eqnarray*}
where the last inequality follows from the fact that $\epsilon$ is chosen according to~\eqref{para-choice}.\\
Thus, $\underline{H}\succeq 0$. Therefore, the convergence condition of Theorem~\ref{theo:general-convergence} is a sufficient but not necessary condition of that of Proposition~\ref{prop:choice-core}.

\end{proof}

The subproblems and convergent conditions of VAPP-AL with some important Jacobian regularization quadratic core functions are summarized in Table~\ref{table}.

\begin{table}[!htp]
\begin{center}
\begin{tabular}{|l|l|l|l|}
\hline
\multicolumn{2}{|l|}{VAPP for (P)}&Subproblem (AP$^k$)& Convergent condition\\
\hline
\tabincell{l}{Proximal\\Jacobian\\VAPP\\(PJVAPP)}&
\tabincell{l}{for all $i$,\\$\theta_i>0$,\\$\alpha_i>0$,\\$\epsilon=1$}&
\tabincell{l}{$\langle \nabla G(u^k),u\rangle+J(u)+\langle p^k,Au-b\rangle$\\$+\sum_{i=1}^{N}\frac{\theta_i}{2}\|A_iu_i+\sum\limits_{j\neq i}A_ju_j^k-b\|^2+\sum_{i=1}^{N}\frac{\alpha_i}{2}\|P_i(u_i-u_i^k)\|^2$}&
\tabincell{l}{$\underline{H}\succeq 0$, $0<\rho\leq(1+\delta)\gamma$\\$A_{i}$ has full column rank\\ or \\$P_{i}$ has full column rank\\ for all $i$}\\
\hline
\tabincell{l}{Linear\\Jacobian\\VAPP\\(LJVAPP)}&
\tabincell{l}{for all $i$,\\$\theta_i>0$,\\$\alpha_i=0$}&
\tabincell{l}{$\langle \nabla G(u^k),u\rangle+J(u)+\langle p^k,Au-b\rangle$\\
$+\langle\gamma(1-\frac{1}{\epsilon})(Au^k-b),Au\rangle+\sum_{i=1}^{N}\frac{\theta_i}{2}\|A_iu_i+\sum\limits_{j\neq i}A_ju_j^k-b\|^2$}&
\tabincell{l}{$\underline{H}\succeq 0$, $0<\rho\leq(1+\delta)\gamma$\\$A_{i}$ has full column rank\\ for all $i$}\\
\hline
\tabincell{l}{Linear\\Proximal\\VAPP\\(LPVAPP)}&
\tabincell{l}{for all $i$,\\$\theta_i=0$,\\$\alpha_i>0$}&
\tabincell{l}{$\langle \nabla G(u^k),u\rangle+J(u)+\langle p^k,Au-b\rangle$\\
$+\langle\gamma(Au^k-b),Au\rangle+\sum_{i=1}^{N}\frac{\alpha_i}{2\epsilon}\|P_i(u_i-u_i^k)\|^2$}&
\tabincell{l}{$\underline{H}\succeq 0$, $0<\rho\leq(1+\delta)\gamma$\\$P_{i}$ has full column rank\\ for all $i$}\\
\hline
\end{tabular}
\end{center}
\caption{Implementation of VAPP with Jacobian regularization quadratic core functions $K^k(\cdot)$\label{table}}
\end{table}
Since the separable programming (SP) is a special case of (P) with $J(u) = 0$, the convergent conditions in Table~\ref{table} can also be applied to (SP).
Moreover, the following remark shows that the VAPP-AL framework may include some existent Jacobian type augmented Lagrangian decomposition methods as special cases.
\begin{remark}
First note that, if $\epsilon=1$, $\theta_{i}=\gamma=\rho$, $\alpha_i =1$,
$\bar{P}_{i}=P_{i}^{\top}P_{i}$ for all $i=1,\dots, N$, and $\bar{P}_{i}\succ\gamma(\frac{N}{2-\rho/\gamma}-1)A_i^{\top}A_i$, then $\underline{H}\succeq 0$, PJVAPP is actually the PJADM in~\cite{DengLaiPengYin-2014}.\\
Second note that, if $\epsilon=1$, $\theta_{i}=\gamma=\rho$, $\alpha_i =s\gamma$,
$P_{i}=A_{i}$ for all $i=1,\dots, N$, and $s\geq N-1$, then $\underline{H}\succeq 0$ PJVAPP is actually the Algorithm (1.6) in~\cite{HeXuYuan16}.
\end{remark}
To illustrate the convergence rate of VAPP-AL with Jacobian regularization quadratic core functions, we quote a lemma in~\cite{DengLaiPengYin-2014}, which is useful to establish the $o(1/k)$ convergence rate in this paper.
\begin{lemma}\label{lemma:4}
\noindent If a sequence $\{a_{k}\}\subseteq R$ obeys: (1) $a_{k}\geq 0$; (2)$\sum_{k=1}^{\infty} a_{k}<+\infty$; (3) $a_{k}$ is monotonically non-increasing, then we have $a_{k}=o(1/k)$.
\end{lemma}

Inspired by the above lemma, the key of our analysis is the monotonicity of sequence $\|u^{k}-u^{k+1}\|_{\overline{H}}^{2}+\frac{1}{\rho}\|p^{k}-p^{k+1}\|^{2}$, which is proved in
the following lemma.
\begin{lemma}\label{lemma:5}
\noindent Suppose the sequence $\{u^k, p^k\}$ is generated by VAPP-AL, $G(\cdot)$ is twice differentiable convex function. If $\underline{H}\succeq 0$ and $0<\rho\leq(1+\delta)\gamma\leq2\gamma$, where $0<\delta\leq 1$, then it holds that
\begin{eqnarray}
\|u^{k}-u^{k+1}\|_{\overline{H}}^{2}+\frac{1}{\rho}\|p^{k}-p^{k+1}\|^{2}\leq \|u^{k-1}-u^{k}\|_{\overline{H}}^{2}+\frac{1}{\rho}\|p^{k-1}-p^{k}\|^{2}
\end{eqnarray}
\end{lemma}
\begin{proof}  To simplify the notation, we let
$$
\Delta u^{k+1}=u^{k}-u^{k+1}\;\mbox{and}\;\Delta p^{k+1}=p^{k}-p^{k+1}.
$$
Consequently, $\Delta u_{i}^{k+1}=u_{i}^{k}-u_{i}^{k+1}$, $i=1,\ldots,N$. Recall that the subproblem of VAPP-AL is given by~\eqref{eq:APk}
and from the optimality condition it follows that for $i=1,\ldots,N$ there exits a point $y_i^{k+1} \in \partial J_{i}(u_{i}^{k+1})$ satisfying
\begin{eqnarray*}
\left( \frac{1}{\epsilon}\left(\nabla_{u_i} K^{k}(u^{k+1})- \nabla_{u_i} K^{k}(u^{k}) \right) + \nabla_{u_{i}} G(u^{k}) + A_{i}^{T}\left( p^{k}+\gamma(Au^{k}-b)\right) + y_i^{k+1} \right)^{T}(u_i - u_i^{k+1}) \ge 0,\;\forall \; u_i \in U_i.
\end{eqnarray*}
Then for any $i=1,\ldots,N$ and $u_i \in U_i$,
plugging the gradient of the core function that defined in~\eqref{eq:ProximalJacobicore} into the above formula yields
\begin{eqnarray}\label{optimality-formula-k}
 \left( - \frac{\theta_i}{\epsilon}A_{i}^{T}A_{i}\Delta u_{i}^{k+1} - \frac{\alpha_i}{\epsilon}P_{i}^{T}P_{i}\Delta u_{i}^{k+1} + \nabla_{u_{i}} G(u^{k}) + A_{i}^{T}\left(p^{k}+\gamma(Au^{k}-b)\right)+ y_i^{k+1} \right)^{T}  (u_i - u_i^{k+1}) \ge 0.
\end{eqnarray}
Repeating the above argument for the $k-1$-th iteration gives that
\begin{eqnarray}\label{optimality-formula-k-1}
 \left( - \frac{\theta_i}{\epsilon}A_{i}^{T}A_{i}\Delta u_{i}^{k} - \frac{\alpha_i}{\epsilon}P_{i}^{T}P_{i}\Delta u_{i}^{k} + \nabla_{u_{i}} G(u^{k-1}) + A_{i}^{T}\left(p^{k-1}+\gamma(Au^{k-1}-b)\right)+ y_i^{k} \right)^{T}  (u_i - u_i^{k}) \ge 0.
\end{eqnarray}

Moreover, since $J(\cdot)$ is convex, for any $y_i^k \in \partial J_{i}(u_{i}^{k})$ and $y_i^{k+1} \in \partial J_{i}(u_{i}^{k+1})$
$$
\langle y_i^{k} - y_i^{k+1}, u_{i}^{k}-u_{i}^{k+1}\rangle\geq 0.
$$
Therefore, by taking $u_i = u^k_i$ in~\eqref{optimality-formula-k} and $u_i = u^{k+1}_i$ in~\eqref{optimality-formula-k-1}, it holds that
\begin{eqnarray*}
0 & \le &\sum_{i=1}^{N}\langle y_i^k -y_i^{k+1}, u_{i}^{k}-u_{i}^{k+1}\rangle\\
&\le & \sum_{i=1}^{N}\frac{\theta_{i}}{\epsilon}\langle A_{i}(\Delta u_{i}^{k}-\Delta u_{i}^{k+1}), A_{i}\Delta u_{i}^{k+1}\rangle + \sum_{i=1}^{N}\frac{\alpha_i}{\epsilon}\langle P_{i}(\Delta u_{i}^{k}-\Delta u_{i}^{k+1}), P_{i}\Delta u_{i}^{k+1}\rangle \\
&& -\langle \nabla G(u^{k-1})- \nabla G(u^{k}), \Delta u^{k+1}\rangle - \langle \Delta p^{k}, A\Delta u^{k+1}\rangle -\gamma \langle A\Delta u^{k}, A\Delta u^{k+1}\rangle \\
&=&-\|\Delta u^{k+1}\|_{H}^{2}+(\Delta u^{k})^{T} \tilde{H}^{k} \Delta u^{k+1} -\langle \Delta p^{k},A\Delta u^{k+1}\rangle,\\
\end{eqnarray*}
where the equality follows as
$$ \nabla G(u^{k-1})- \nabla G(u^{k})= \int_{0}^{1}d\, \nabla G(u^{k}+\tau(u^{k-1}-u^{k})) = \int_{0}^{1}\nabla^2G(u^{k}+\tau(u^{k-1}-u^{k}))d\tau(u^{k-1}-u^{k}).$$
Since $\overline{H} \succeq  \tilde{H}^{k} \succeq \underline{H} \succeq 0$, the above inequality implies that
\begin{eqnarray}\label{eq:Adeltau}
2\langle \Delta p^{k},A\Delta u^{k+1}\rangle &\leq& -2\|\Delta u^{k+1}\|_{H}^{2}+2(\Delta u^{k})^{T} \tilde{H}^{k} \Delta u^{k+1} \nonumber\\
                                            &\leq& -2\|\Delta u^{k+1}\|_{H}^{2}+\|\Delta u^{k}\|_{\tilde{H}^{k}}^{2}+\|\Delta u^{k+1}\|_{\tilde{H}^{k}}^{2}\nonumber \\
                                            &\le& -2\|\Delta u^{k+1}\|_{H}^{2}+\|\Delta u^{k}\|_{\overline{H}}^{2}+\|\Delta u^{k+1}\|_{\overline{H}}^{2}
\end{eqnarray}
Observing $\Delta p^{k+1}=\Delta p^{k} + \rho A\Delta u^{k+1}$ is the dual update of VAPP-AL, it follows that
\begin{eqnarray*}
\frac{1}{\rho}\|\Delta p^{k+1}\|^{2}-\frac{1}{\rho}\|\Delta p^{k}\|^{2}&=&2\langle \Delta p^{k},A\Delta u^{k+1}\rangle+\rho \|A\Delta u^{k+1}\|^{2}\\
&\leq& -2\|\Delta u^{k+1}\|_{H}^{2}+\|\Delta u^{k}\|_{\overline{H}}^{2}+\|\Delta u^{k+1}\|_{\overline{H}}^{2}+\rho \|A\Delta u^{k+1}\|^{2}.
\end{eqnarray*}
As a result,
\begin{eqnarray*}
&&\left( \|\Delta u^{k+1}\|_{\overline{H}}^{2}+\frac{1}{\rho}\|\Delta p^{k+1}\|^{2}\right) - \left(\|\Delta u^{k}\|_{\overline{H}}^{2}+\frac{1}{\rho}\|\Delta p^{k}\|^{2}\right)\\
&\leq& -2\|\Delta u^{k+1}\|_{H}^{2}+2\|\Delta u^{k+1}\|_{\overline{H}}^{2}+\rho \|A\Delta u^{k+1}\|^{2}\\
&=& -\|\Delta u^{k+1}\|_{(2\gamma-\rho)A^{\top}A}^{2}\\
&\leq& 0,
\end{eqnarray*}
the last inequality is due to the choice of $\rho$ in~\eqref{para-choice} and the proof is complete.
\end{proof}

With those preparations in hand, the convergence rate of our algorithm readily follows.
\begin{theorem}\label{thm:convergence-rate}
\noindent Suppose Assumption \ref{assump1} holds, the core function is defined in~\eqref{eq:ProximalJacobicore}, and
for each index $i=1,\dots, N$, we either have $A_{i}$ has full column rank with $\theta_i > 0$ or $P_{i}$ has full column rank with $\alpha_i >0$.
Moreover, assume that $\underline{H}\succeq 0$ and
 $G(\cdot)$ is convex, twice differentiable and its hessian matrix is bounded. Suppose the sequence $\{u^k,p^k \}$ is generated by VAPP-AL and the values of the parameters $\rho$ and $\epsilon$
are chosen based on~\eqref{para-choice}. Then
 it follows that for any integer $t>0$ we have $\|u^{t}-u^{t+1}\|_{\overline{H}}^{2}=o(1/t)$ and $\|p^{t}-p^{t+1}\|^{2}=o(1/t)$.
\end{theorem}
\begin{proof}
\noindent Due to the previous discussion and the choice of $\epsilon$, $\rho$, Theorem~\ref{theo:general-convergence} follows. Therefore,
\begin{eqnarray*}
    && \Lambda^{k+1}(u^{k+1},p^{k+1})-\Lambda^k(u^{k},p^{k})\nonumber\\
&\leq&  \frac{1}{2}\left(\epsilon\left(B_{G}+\gamma \lambda_{max}(A^{\top}A)\right)-\beta^k \right)\|u^{k}-u^{k+1}\|^{2} +\frac{ \epsilon}{2}(\rho-(1+\delta) \gamma )\|Au^{k+1}-b\|^{2},
\end{eqnarray*}
and $\lim\limits_{k \to \infty} \Lambda^k(u^{k},p^{k}) = 0$. Since $\frac{1}{2}\left(\epsilon\left(B_{G}+\gamma \lambda_{max}(A^{\top}A)\right)-\beta^k \right) \le 0$ and $\frac{ \epsilon}{2}(\rho-(1+\delta)\gamma )\le 0$, there exist positive numbers $\eta_1$ and $\eta_2$ such that
$$
\frac{1}{2}\left(\epsilon\left(B_{G}+\gamma \lambda_{max}(A^{\top}A)\right)-\beta^k \right) I \preceq - \eta_1 \overline{H}\quad\mbox{and}\quad\frac{ \epsilon}{2}(\rho-(1+\delta) \gamma )I \le - \frac{\eta_2}{\rho} I.
$$
By letting $\eta = \min\{\eta_1, \eta_2\}$, one has that
\begin{equation}\label{eq:sum}
\Lambda^{k+1}(u^{k+1},p^{k+1})-\Lambda^k(u^{k},p^{k}) \le -\eta\left( \|u^{k}-u^{k+1} \|_{\overline{H}}^{2}+\frac{1}{\rho}\|p^{k}-p^{k+1}\|^{2} \right)
\end{equation}
Summing (\ref{eq:sum}) over $k$ and taking limit yields
\begin{eqnarray*}
\sum_{k=1}^{\infty}\left( \|u^{k}-u^{k+1}\|_{\overline{H}}^{2}+\frac{1}{\rho}\|p^{k}-p^{k+1}\|^{2} \right)\le \frac{1}{\eta} \Lambda^0(u^{0},p^{0})<\infty.
\end{eqnarray*}
\noindent Moreover Lemma \ref{lemma:5} implies that the sequence $\{ \|u^{k}-u^{k+1}\|_{\overline{H}}^{2}+\frac{1}{\rho}\|p^{k}-p^{k+1}\|^{2} \}$ is non-increasing.
Finally, the conclusion follows by combining those results with Lemma \ref{lemma:4}.
\end{proof}

\section{Conclusion}
In this paper, we propose a new VAPP-AL framework that can allow the core function of standard APP to be different at each iteration.
We illustrate that how some known algorithms (especially some variants of ADMM) can be considered as specializations and solved by our framework.
We prove that under some mild conditions,
the solution sequence generated by VAPP-AL converges.
Moreover the the convergence rate of $O(1/t)$ are provided in both ergodic and non-ergodic sence.
When the core function is specialized to be quadratic, we show that an even lower iteration complexity of $o(1/t)$ can be achieved. In addition,
our framework can handle a convex problem with nonseparable
objective and multi-blocks coupled linear constraints.


\begin{thebibliography} {99}
\bibitem{Bert}
Bertsekas, D.P. (1999). {\sl Nonlinear Programming}. Athena Scientific, Belmont Massachusetts.
\bibitem{Boser}
Boser, B. E., Guyon, I. M., \& Vapnik, V. N. (1992, July). A training algorithm for optimal margin classifiers. {\sl In Proceedings of the fifth annual workshop on Computational learning theory} (pp. 144-152). ACM.
\bibitem{Boyd2011}
Boyd, S., Parikh, N., Chu, E., Peleato, B., \& Eckstein, J. (2011). Distributed optimization and statistical learning via the alternating direction method of multipliers. {\sl Foundations and Trends in Machine Learning, 3}(1), 1-122.
\bibitem{Breg}
Bregman, L.M. (1967). The relaxation method of finding the common point of convex sets and its application to the solution of problems in convex programming. {\sl USSR computational mathematics and mathematical physics, 7}(3), 200-217.
\bibitem{CaiHanYuan15}
Cai, X., Han, D., \& Yuan, X. (2014). The direct extension of ADMM for three-block separable convex minimization models is convergent when one function is strongly convex. {\sl Optimization Online}.
\bibitem{CelaHamam92}
$\c{C}$ela, A. S., \& Hamam, Y. (1992). Optimal motion planning of a multiple-robot system based on decomposition coordination. {\sl Robotics and Automation, IEEE Transactions on, 8}(5), 585-596.
\bibitem{ChenHeYeYuan13}
Chen, C., He, B., Ye, Y., \& Yuan, X. (2014). The direct extension of ADMM for multi-block convex minimization problems is not necessarily convergent. {\sl Mathematical Programming,} 1-23.
\bibitem{ChenTeboulle93}
Chen, G., \& Teboulle, M. (1993). Convergence analysis of a proximal-like minimization algorithm using Bregman functions. {\sl SIAM Journal on Optimization, 3}(3), 538-543.
\bibitem{ChenTeboulle94}
Chen, G., \& Teboulle, M. (1994). A proximal-based decomposition method for convex minimization problems. {\sl Mathematical Programming, 64}(1-3), 81-101.
\bibitem{CenZen92}
Censor, Y., \& Zenios, S. A. (1992). Proximal minimization algorithm withD-functions. {\sl Journal of Optimization Theory and Applications, 73}(3), 451-464.
\bibitem{Cohen78}
Cohen, G. (1978). Optimization by decomposition and coordination: a unified approach. {\sl Automatic Control, IEEE Transactions on, 23}(2), 222-232.
\bibitem{Cohen80}
Cohen, G. (1980). Auxiliary problem principle and decomposition of optimization problems. {\sl Journal of optimization Theory and Applications, 32}(3), 277-305.
\bibitem{CohenZ}
Cohen, G., \& Zhu, D. L. (1984). Decomposition coordination methods in large scale optimization problems. The nondifferentiable case and the use of augmented Lagrangians. {\sl Advances in large scale systems, 1}, 203-266.
\bibitem{Distopt1}
Contreras, J., Losi, A., Russo, M., \& Wu, F. F. (2000). DistOpt: A software framework for modeling and evaluating optimization problem solutions in distributed environments. {\sl Journal of Parallel and Distributed Computing, 60}(6), 741-763.
\bibitem{Cortes}
Cortes, C., \& Vapnik, V. (1995). Support-vector networks. {\sl Machine learning, 20}(3), 273-297.
\bibitem{CuiLiSunToh-2015}
Cui, Y., Li, X., Sun, D., \& Toh, K. C. (2015). On the convergence properties of a majorized ADMM for linearly constrained convex optimization problems with coupled objective functions. {\sl arXiv preprint arXiv:1502.00098}.
\bibitem{DengLaiPengYin-2014}
Deng, W., Lai, M. J., Peng, Z., \& Yin, W. (2013). Parallel multi-block ADMM with O (1/k) convergence. {\sl arXiv preprint arXiv:1312.3040}.
\bibitem{Donoho-1992}
Donoho, D. L. (1992). Denoising by Soft-Thresholding. {\sl Dept. of Statistics}.
\bibitem{FortinGlowinski}
Fortin, M., \& Glowinski, R. (1983). Chapter III on decomposition-coordination methods using an augmented lagrangian. {\sl Studies in Mathematics and Its Applications, 15}, 97-146.
\bibitem{Gao-Zhang-2015}
Gao, X., \& Zhang, S. (2015). First-Order Algorithms for Convex Optimization with Nonseparate Objective and Coupled Constraints. {\sl Working Paper}.
\bibitem{Glowinski84}
Glowinski, R. (1984). Decomposition-Coordination Methods by Augmented Lagrangian: Applications. In {\sl Numerical Methods for Nonlinear Variational Problems} (pp. 166-194). Springer Berlin Heidelberg.
\bibitem{HeHouYuan13}
He, B., Hou, L., \& Yuan, X. (2013). On full Jacobian decomposition of the augmented Lagrangian method for separable convex programming. {\sl Preprint}.
\bibitem{HeXuYuan16}
He, B., Xu, H. K., \& Yuan, X. (2016). On the proximal Jacobian decomposition of ALM for multiple-block separable convex minimization problems and its relationship to ADMM. {\sl Journal of Scientific Computing, 66}(3), 1204-1217.
\bibitem{Hestenes1969}
Hestenes, M. R. (1969). Multiplier and gradient methods. {\sl Journal of optimization theory and applications, 4}(5), 303-320.
\bibitem{linearSVM}
Keerthi, S. S., \& DeCoste, D. (2005). A modified finite Newton method for fast solution of large scale linear SVMs. {\sl Journal of Machine Learning Research, 6}(Mar), 341-361.
\bibitem{Electro97}
Kim, B. H., \& Baldick, R. (1997). Coarse-grained distributed optimal power flow. {\sl IEEE Transactions on Power Systems, 12}(2), 932-939.
\bibitem{KimBaldick00}
Kim, B. H., \& Baldick, R. (2000). A comparison of distributed optimal power flow algorithms. {\sl Power Systems, IEEE Transactions on, 15}(2), 599-604.
\bibitem{LiMoYuanZhang14}
Li, X., Mo, L., Yuan, X., \& Zhang, J. (2014). Linearized alternating direction method of multipliers for sparse group and fused LASSO models. {\sl Computational Statistics \& Data Analysis, 79}, 203-221.
\bibitem{LinMaZhang15}
Lin, T. Y., Ma, S. Q., \& Zhang, S. Z. (2015). On the sublinear convergence rate of multi-block ADMM. {\sl Journal of the Operations Research Society of China, 3}(3), 251-274.
\bibitem{Distopt2}
Losi, A., \& Russo, M. (2003). On the application of the auxiliary problem principle. {\sl Journal of optimization theory and applications, 117}(2), 377-396.
\bibitem{Ortega}
Ortega, J. M., \& Rheinboldt, W. C. (1970). {\sl Iterative solution of nonlinear equations in several variables} (Vol. 30). Siam.
\bibitem{Platt}
Platt, J. C. (1999). 12 fast training of support vector machines using sequential minimal optimization. {\sl Advances in kernel methods}, 185-208.
\bibitem{Powell1969}
Powell, M. J. D. (1969). A method for nonlinear constraints in minimization problems. R. Fletcher, ed. {\sl Optimization}. Academic Press, London, U.K.
\bibitem{Electro93}
Renaud, A. (1993). Daily generation management at Electricité de France: from planning towards real time. {\sl Automatic Control, IEEE Transactions on, 38}(7), 1080-1093.
\bibitem{Rock72}
Rockafellar, R. T. (2015). {\sl Convex analysis}. Princeton university press.
\bibitem{Rock76}
Rockafellar, R. T. (1976). Augmented Lagrangians and applications of the proximal point algorithm in convex programming. {\sl Mathematics of operations research, 1}(2), 97-116.
\bibitem{ShefiTeboulle14}
Shefi, R., \& Teboulle, M. (2014). Rate of convergence analysis of decomposition methods based on the proximal method of multipliers for convex minimization. {\sl SIAM Journal on Optimization, 24}(1), 269-297.
\bibitem{Tibshirani05}
Tibshirani, R., Saunders, M., Rosset, S., Zhu, J., \& Knight, K. (2005). Sparsity and smoothness via the fused lasso. {\sl Journal of the Royal Statistical Society: Series B (Statistical Methodology), 67}(1), 91-108.
\bibitem{Fusedlassosvm}
Watanabe, T., Scott, C. D., Kessler, D., Angstadt, M., \& Sripada, C. (2014, May). Scalable fused Lasso SVM for connectome-based disease prediction. {\sl In 2014 IEEE International Conference on Acoustics, Speech and Signal Processing (ICASSP)} (pp. 5989-5993). IEEE.
\bibitem{Zhu83}
Zhu, D.L. (1983). {\sl OPTIMISATION SOUS-DIFFERENTIABLE ET METHODES DE DECOMPOSITION} (Doctoral dissertation).
\bibitem{Zhu95}
Zhu, D., \& Marcotte, P. (1995). Coupling the auxiliary problem principle with descent methods of pseudoconvex programming. {\sl European journal of operational research, 83}(3), 670-685.
\bibitem{Zhu03}
Zhu, D. L. (2003). Augmented Lagrangian theory, duality and decomposition methods for variational inequality problems. {\sl Journal of optimization theory and applications, 117}(1), 195-216.





\end{thebibliography}
\end{document}